\documentclass[12pt,english,a4paper]{article}
\usepackage{amssymb}
\usepackage{amsfonts}

\usepackage{amsmath,amscd, amsthm}

\usepackage{euscript}
\usepackage[numbers]{natbib}
\usepackage[latin9]{inputenc}
\usepackage{geometry}
\usepackage{pdfsync}

\usepackage{hyperref}
\usepackage{enumitem}

\usepackage[all]{xy}

\newcommand{\Z}{\mathbb{Z}}
\newcommand{\Q}{\mathbb{Q}}
\newcommand{\R}{\mathbb{R}}
\newcommand{\C}{\mathbb{C}}
\renewcommand{\P}{\mathbb{P}}

\newcommand{\p}{\mathfrak{p}}

\newcommand{\Fix}{\operatorname{Fix}}

\newcommand{\iso}{\stackrel{\sim}{\longrightarrow}}

\renewcommand{\d}{\partial}
\providecommand{\abs}[1]{\left|#1\right|}
\newcommand{\set}[1]{\left\{{#1}\right\}}

\newcommand{\CP}{\mathbb{CP}}

\newcommand{\half}{\frac{1}{2}}
\newcommand{\al}{\alpha}
\newcommand{\conj}{\overline}
\newcommand{\M}{\mathcal{M}}

\newcommand{\Line}{\mathbb{L}}
\newcommand{\ExDiv}{\mathbb{E}}

\newtheorem{thm}{Theorem}[section]
\newtheorem{prop}[thm]{Proposition}
\newtheorem{cor}[thm]{Corollary}
\newtheorem{lem}[thm]{Lemma}

\usepackage{babel}

\begin{document}
\title{The open Gromov-Witten-Welschinger theory of blowups of the projective plane}
\author{Asaf Horev, Jake P. Solomon}
\date{October 2012}
\maketitle

\begin{abstract}
We compute the Welschinger invariants of blowups of the projective plane at an arbitrary conjugation invariant configuration of points. Specifically, open analogues of the WDVV equation and Kontsevich-Manin axioms lead to a recursive algorithm that reconstructs all the invariants from a small set of known invariants. Example computations are given, including the non-del Pezzo case.
\end{abstract}

\pagestyle{plain}

\tableofcontents


\section{Introduction}

The object of this paper is the calculation of the Welschinger invariants of $\CP^2_{r,s}$, the blowup of a projective plane at $r$ real points and $s$ pairs of conjugate points, for arbitrary $r,s$.

In \cite{ko}, Kontsevich and Manin gave a simple recursive formula calculating the Gromov-Witten invariants of $\CP^2$.
The geometric insight behind the formula is a splitting principle, captured by the WDVV equations~\cite{Witten}, and several other properties known as axioms.
The WDVV equations and Kontsevich-Manin axioms were then used by G\"ottsche and Pandharipande~\cite{GP} to give a set of formulae that recursively compute the Gromov-Witten invariants of $\CP^2_r$, the blowup of $\CP^2$ at $r$ points.

A \emph{real symplectic manifold} $(X,\omega,\phi)$ is a  symplectic manifold $(X,\omega)$ together with an anti-symplectic involution
\begin{align*}
\phi: X \to X, \quad \phi^*\omega = -\omega.
\end{align*}
The involution $\phi$ generalizes complex conjugation. For a real strongly semi-positive symplectic manifold of dimension $4$ or~$6,$ Welschinger~\cite{Welschinger,We2} defined invariants based on a signed count of real genus $0$ $J$-holomorphic curves passing through a $\phi$-invariant set of point constraints. In dimension $4$ the sign is given by the parity of the number of isolated real double points and in dimension $6$ by spinor states.
For a real strongly semi-positive symplectic manifold $X$ of dimension $4$ or $6$ such that the fixed point set $\Fix(\phi) \subset X$ is orientable, Cho~\cite{cho} used the moduli space of $J$-holomorphic discs to define open Gromov-Witten invariants.
Independently, the second author~\cite{jake1} defined open Gromov-Witten invariants of general real symplectic manifolds of dimension $4$ and $6$ in arbitrary genus, assuming orientability of $\Fix(\phi)$ only in dimension $6.$ He also showed that the open Gromov-Witten invariants recover Welschinger's invariants in the strongly semi-positive genus $0$ case.

The present paper uses analogues of the WDVV equation and Kontsevich-Manin axioms for open Gromov-Witten invariants due to the second author~\cite{jake2}. He showed that the open WDVV equation and axioms give recursive formulae for the Welschinger invariants of $\CP^2.$ We prove the following theorem for $\CP^2_{r,s}.$
\begin{thm} \label{complete OGW}
The Welschinger invariants of $\CP^2_{r,s}$ with arbitrary real and complex conjugate point constraints are completely determined by
\begin{itemize}
\item the open Kontsevich-Manin axioms (see Subsection \ref{axioms}),
\item the open WDVV equations (see Subsection \ref{open WDVV}),
\item the closed Gromov-Witten invariants of $\CP^2_{r+2s}$ (computed by~\cite{GP}), and
\item a finite set of known initial values (see Section \ref{initial calculations}).
\end{itemize}
\end{thm}
Our proof will consist of a set of recursive relations and an explicit algorithm for computing the open Gromov-Witten invariants of $X$ starting from a finite set of initial values.
Theorem \ref{complete OGW} is a direct consequence of Section \ref{proof of recursion}, and the initial values are computed in Section \ref{initial calculations}. A selection of values obtained from the algorithm is given in Section~\ref{sec:example}.

The Welschinger invariants of $\CP^2_{r,s}$ for $r+2s \leq 6$ with purely real point constraints were studied previously by Itenberg, Kharlamov and Shustin~\cite{IKS2,IKS3,IKS4,IKS5,IKS6,IKS7,IKS8}. They used techniques of tropical geometry, and a real analogue of the Caporaso-Harris formula. In~\cite{BM09}, Brugall\'e and Mikhalkin computed the Welschinger invariants of $\C P^2_{r,s}$ for $r+2s \leq 3$ with arbitrary real and complex conjugate pairs of point constraints. They used the tropical technique of floor diagrams. Our results agree with the previously obtained results in most cases.

\subsection{Acknowledgements}
The authors would like to thank Emmanuel Farjoun for helpful conversations. The authors were partially supported by United States - Israel Binational Science Foundation grant No. 2008314 and Marie Curie International Reintegration Grant No. 239381.

\section{\texorpdfstring{Gromov-Witten invariants and Kontsevich-\\Manin axioms}{Gromov-Witten invariants and Kontsevich-Manin axioms}}
\label{GW axioms}

Genus zero Gromov-Witten invariants of a symplectic manifold $(X,\omega)$ of dimension $2n$ are a set of maps $ GW_d: H^*(X;\Q)^{\otimes k} \to \Q$, one for each $d\in H_2(X;\Z)$ and $k=0,1,2,\ldots$
Intuitively, the Gromov-Witten invariant $ GW_d(\gamma_1,\ldots,\gamma_k)$ counts the number of $J$-holomorphic spheres intersecting submanifolds $R_1,\ldots, R_k \subset X$ that represent the Poincar\'e duals of $\gamma_1,\ldots,\gamma_k$.

The Gromov-Witten invariants admit a set of basic properties, the Kontsevich-Manin axioms \cite{ko}.
\begin{prop}(The Kontsevich-Manin axioms.)\\
\begin{enumerate}
\item Symmetry:
 $GW_d(\gamma_1 \otimes \ldots \otimes \gamma_k)$ is $\Z_2$ graded symmetric with respect to permutations of the constraints $\gamma_1,\ldots,\gamma_k$.
\item Grading: $GW_d (\gamma_1\otimes \cdots \otimes \gamma_k)=0$ unless $ \Sigma \abs{\gamma_i} = 2 \dim_\R X + 2 c_1(d) + 2n -6 $.
\item Fundamental class: Let $t_0 \in H^0(X)$ be the unit. Then
\[
GW_d (\gamma_1 \otimes \ldots \otimes \gamma_k \otimes t_0) =  \begin{cases}  \int_X \gamma_1 \wedge \gamma_2 & \text{if d=0,k=2,} \\ 0  & \text{otherwise.} \end{cases}
\]
\item Divisor: For $ \gamma \in H^2 (X) $, we have
\[GW_d (\gamma_1 \otimes \cdots \otimes \gamma_k \otimes \gamma) = \left( \int_d \gamma \right) \cdot GW_d (\gamma_1 \otimes \cdots \otimes \gamma_k).
\]
\end{enumerate}
\end{prop}

An expository account of genus $0$ Gromov-Witten invariants and the Kontse\-vich-Manin axioms can be found in~\cite{MS} Chapter $7.$

\subsection{Open Gromov-Witten invariants}

For a real symplectic manifold of dimension $2n$ with $n=2,3$, the second author~\cite{jake1} defined Gromov-Witten type invariants using intersection theory on the moduli space of $J$-holomorphic discs. In this paper the term \emph{open Gromov-Witten invariant} will refer to this definition. In the following paragraphs we recall the relevant material from~\cite{jake1}.

Let $(X,\omega,\phi)$ be a real symplectic manifold and let $L=\Fix(\phi) \subset X$ be the Lagrangian submanifold given by the fixed points of $\phi.$
The \emph{$\phi$-invariant} homology groups with rational coefficients $H^{\phi}_{2i}(X;\Q)$,
$H^{\phi}_{2i}(X,L;\Q),$ are defined as the $(-1)^i$ eigenspaces of
\begin{align*}
\phi_*:H_{2i}(X;\Q) \to H_{2i}(X;\Q), \quad \phi_*:H_{2i}(X,L;\Q) \to H_{2i}(X,L;\Q)
\end{align*} respectively.
There is a unique $\phi_*$-invariant projection $H_{2i}(X;\Q) \to H^{\phi}_{2i}(X;\Q).$ The $\phi$-invariant homology group $H^{\phi}_{2i}(X;\Z) \subseteq H^{\phi}_{2i}(X;\Q)$ is the image of the composition
\[ H_{2i}(X;\Z) \to H_{2i}(X;\Q) \to H^{\phi}_{2i}(X;\Q). \]
The relative $\phi$-invariant homology groups $H^{\phi}_{2i}(X,L;\Z) \subseteq H^{\phi}_{2i}(X,L;\Q)$ are defined in the same way.
Similarly, we define the \emph{$\phi$-invariant} cohomology group $H^{2i}_\phi(X;\Q)$ to be the $(-1)^i$ eigenspace of \(\phi^*: H^{2i}(X;\Q) \to H^{2i}(X;\Q) \).
The cup product makes $H^*_\phi(X;\Q)$ into a ring.

For simplicity, in the following overview, we restrict to the case $n = 2,$ so $X$ is real $4$ dimensional. The open Gromov-Witten invariants of $(X,\omega,\phi)$ are maps
\[ OGW_{d,k} : H^*_\phi(X;\Q)^{\otimes l}\to \Q, \]
one for every choice of $\phi$-invariant homology class $d\in H^\phi_2(X,L;\Z)$ and non-negative integers $k,l\geq 0$.

In the following, we work with an $\omega$-tame almost complex structure $J.$ We denote by $\M^D_{k,l}(X/L,d)$ the moduli space of $J$-holomorphic discs
\[
u:(D^2, \d D^2) \to (X,L)
\]
such that $u_*([D^2,\partial D^2]) \in H_2(X,L;\Z)$ projects to $d\in H^\phi_2(X,L;\Z)$, with $k$ marked boundary points and $l$ marked interior points, modulo reparametrization.
We denote the evaluation maps of the boundary marked points and interior marked points by
 \begin{align*}
 evb_i:\M^D_{k,l}(X/L,d) \to L, \quad i=1,\ldots,k \\
 evi_j: \M^D_{k,l}(X/L,d) \to X,\quad j=1,\ldots, l
\end{align*}
See~\cite{Frauen} and~\cite{jake1} for details.

Throughout this paper, we denote by $\mu : H_2^\phi(X;\Z) \to \Z$ the Maslov index. The following theorem proved in~\cite{jake1} gives a relative orientation of the moduli space of discs.
\begin{thm} \label{relative orientation}
Assume $k=\mu(d)+1 \mod 2$. A $Pin^-$ structure $\mathfrak{p}_L$ on $L$ and an orientation on $L$ if $L$ is orientable determine up to homotopy an isomorphism
\[
\det(T \M^D_{k,l}(X/L,d)) \iso \otimes_{i=1}^{k} evb_i^* \det(TL).\]
\end{thm}

We fix a $Pin^-$ structure $\mathfrak{p}_L$ on $L$ and an orientation on $L$ if $L$ is orientable.

Let $\gamma_1,\ldots,\gamma_l \in H^*_\phi(X;\Q)$ be homogeneous $\phi$-invariant cohomology classes.
Choose $R_i:M_i \to X$ smooth maps representing the Poincar\'e duals of $\gamma_i$, where $M_i$ are smooth oriented manifolds of dimension $2n-\abs{\gamma_i}$.
Choose also $k$ points on $L$ represented by $S_i: pt \to L,$ for $i=1,\ldots,k$.
Let
\[
ev:  \M^D_{k,l}(X/L,d)\to X^l \times L^k,
\]
be the product
\[
ev = evi_1 \times \cdots \times evi_l \times evb_1 \times \cdots \times evb_k,
\]
and let
\[
R:  M_1 \times \cdots \times M_k  \times pt \times \cdots \times pt \to X^l \times L^k,
\]
be the product
\[
R= R_1 \times \cdots \times R_l \times S_1 \times \cdots \times S_k.
\]
Generically, $R$ and $ev$ are transverse, so their pull-back
\begin{align*}
\xymatrix{
P  \ar[r]^<<<<<<<{p_1} \ar[d]^{p_2} & M_1 \times \cdots \times M_n \times pt \times \cdots \times pt \ar[d]^{R} \\
\M^D_{k,l}(X/L,d)\ar[r]^{ev} & X^l \times L^k
}
\end{align*}
is a smooth manifold of dimension
\begin{align*}
\dim P &= \dim \M^D_{k,l} (X/L,d) + \dim(M_1 \times \cdots \times M_k)  - \dim (X^l \times L^k)\\
& = (n + \mu(d) + k + 2l -3 ) - \Sigma_i \abs{\gamma_i}-nk.
\end{align*}

To define open Gromov-Witten invariants we need the following property of relative orienations.
\begin{lem}
\label{pull back relative orientation}
Let $f:X \to A,\quad g:Y \to A$ be smooth transverse maps.
A relative orientation \( \det (TX) \iso f^*\det(TA) \) of $f$ induces a relative orientation \( \det (T (X \times_A Y) ) \iso p^* \det (TY) \) of the projection \( p:X \times_A Y \to Y \).
\end{lem}

Assume for the next paragraph that the conditions of Theorem \ref{relative orientation} hold.
Then we have a relative orientation
\[ \det(T \M^D_{k,l}(X/L,d)) \iso \otimes_{i=1}^{k} evb_i^* \det(TL), \]
so by Lemma \ref{pull back relative orientation} we have a relative orientation
\[ \det(T P  ) \iso p_{1}^* \det(T (M_1 \times \cdots \times M_k  \times pt \times \cdots \times pt)). \]
Therefore, since the manifold $M_1 \times \cdots \times M_k$ is oriented, so is $P.$ We denote the orientation by $o.$

We define  $OGW_{d,k} (\gamma_1 \otimes \cdots \otimes \gamma_k)$ as follows.
\begin{itemize}
\item If $\dim(P) \neq 0$, set $OGW_{d,k} (\gamma_1 \otimes \cdots \otimes \gamma_k)=0$.
\item If $\dim(P) = 0$ then $k = \mu(d) +1 \mod{2}$ as required by Theorem~\ref{relative orientation}, and $P$ is oriented. So, we define $OGW_{d,k} (\gamma_1 \otimes \cdots \otimes \gamma_k) =  \sum_{p \in P} o(p)$.
\end{itemize}
We say an almost complex structure $J$ is \emph{$\phi$-invariant} if $\phi^* J = -J$. For generic $\phi$-invariant $J,$ the last sum is finite and well defined by the following theorem from~\cite{jake1}.
\begin{thm}
The numbers $OGW_{d,k} (\gamma_1\otimes \cdots \otimes \gamma_k)$ are finite and independent of the choice of generic $\phi$-invariant $J$, the points $S_i$ on $L$ and the maps $R_i:M_i \to X$ representing the duals of $\gamma_i$.
\end{thm}

\subsection{Axioms for open Gromov-Witten invariants}
\label{axioms}
The open Gromov-Witten invariants satisfy the following properties analogous to the Kontsevich-Manin axioms. We use the fact that the natural map
\begin{equation}\label{eq:isoj}
H_2^\phi(X;\Q) \to H_2^\phi(X,L,\Q)
\end{equation}
is an isomorphism. This follows from the exact sequence of the pair $(X,L)$ because $\phi$ acts trivially on $H_2(L).$
\begin{prop}(The open Gromov-Witten axioms.)\\
\begin{enumerate}
\item Symmetry:
 $OGW_{d,k}(\gamma_1 \otimes \ldots \otimes \gamma_l)$ is $\Z_2$-graded symmetric with respect to permutations of the constraints $\gamma_1, \ldots, \gamma_l$.
\item Grading:  $OGW_{d,k} (\gamma_1\otimes \cdots \otimes \gamma_l)=0$ unless
\begin{align*}
\Sigma \abs{\gamma_i}  + k \dim_\R L=  \dim_\C X + \mu(d) + 2l + k -3. \label{open grading}
\end{align*}
\item Fundamental class: Let $\tau_0 \in H^0_\phi (X)$ be the unit. Then
\[
OGW_{d,k} (\gamma_1 \otimes \ldots \otimes \gamma_l \otimes \tau_0 ) =  \begin{cases}  1 & \text{if d=0, k=1 and l=0}, \\ 0 & \text{otherwise.} \end{cases}
\]

\item Divisor: for $\gamma \in H^2_\phi (X)$, we have
\[
OGW_{d,k} (\gamma_1 \otimes \cdots \otimes \gamma_l \otimes \gamma) = \left( \int_d \gamma \right) \cdot OGW_{d,k} (\gamma_1 \otimes \cdots \otimes \gamma_l),\]
where $\int_d\gamma$ is defined by identifying $d$ with an element of $H_2^\phi(X)$ using isomorphism~\eqref{eq:isoj}.
\end{enumerate}
\end{prop}

\subsection{WDVV equations}
The WDVV equations are a concise way of writing down further relations among Gromov-Witten invariants that come from non-linear gluing theory. In this section we give the relevant definitions and formulate the WDVV equations for closed and open Gromov-Witten invariants.

Let $(X,\omega)$ be a symplectic manifold of dimension $4.$ The closed Gromov-Witten potential of $X$ is a formal power series encoding the closed Gromov-Witten invariants of $X.$
Let $t_0,\ldots,t_m \in H^*(X;\Q)$ be a basis for the rational cohomology of $X$.
Let \( w = (w_{0},\ldots,w_{m}) \) be a collection of formal variables, and let
\[
\delta_{w} = w_{0} t_{0} + \cdots + w_{m} t_{m}
\]
be the corresponding formal cohomology class. Let $T$ be an additional formal variable. The closed Gromov-Witten potential is the formal power series
\begin{equation}\label{eq:Phi}
\Phi(w) = \sum_{d,n} \frac{T^{\int_d\omega}}{n!} GW_{d}(\delta_{w}^{\otimes n}).
\end{equation}
The variable $T$ is necessary to take care of convergence issues.
We denote by $(g_{ij})$ the matrix with entries
\[
g_{ij} = \begin{cases} \int_{X} t_{i} \wedge t_{j} & \text{if } t_{i} \wedge t_{j} \in H^{4}(X), \\ 0 & \text{otherwise,} 	 \end{cases}
\]
and we denote by $(g^{ij})$ the inverse matrix.

In the following theorems, we use the Einstein summation convention. Mathematical proofs of the following theorem first appeared in~\cite{RT} and~\cite{MSO}.
\begin{thm}
Let $i,j,k,l \in \set{0,\ldots,m}$. Then the WDVV equation holds,
\[
 \d_i \d_j \d_{\nu} \Phi \cdot g^{\nu \mu} \cdot \d_{\mu} \d_k \d_l  \Phi
=  \d_j \d_k \d_{\nu} \Phi \cdot g^{\nu \mu} \cdot \d_{\mu} \d_i \d_l  \Phi,
\]
where $\d_i$ differentiation with respect to $w_i$.
\end{thm}

For a real symplectic manifold $(X,\omega,\phi)$ of dimension $\dim X=4,$ the open Gromov-Witten invariants and closed Gromov-Witten invariants together satisfy the open WDVV equations. Let $\tau_0,\ldots,\tau_m \in H^*_\phi(X;\Q)$ be a basis.
Let \( w = (w_{0},\ldots,w_{m}) \) be a collection of formal variables, and let
\[
\delta_{w} = w_{0} \tau_{0} + \cdots + w_{m} \tau_{m}
\]
be the corresponding formal cohomology class. We still define the closed Gromov-Witten potential $\Phi(w)$ by equation~\eqref{eq:Phi}, but now $\delta_w$ takes values only in the subspace $H^*_\phi(X;\Q) \subset H^*(X;\Q).$

Let $T$ and $u$ be additional formal variables. The open Gromov-Witten potential is the formal power series
\[
\Omega(w,u) = \sum_{d,k,l} \varepsilon(d) \frac{T^{\int_d\omega}u^{k}}{k! l!} OGW_{(d,\al),k}(\delta_{w}^{\otimes l}),
\]
where $\varepsilon$ is given by
\[
\varepsilon(d) = \begin{cases} +1 & \text{ if } \mu(d) = 0 \mod 2,  \\  \sqrt{-1} & \text{ if } \mu(d) =1 \mod 2. \end{cases}
\]
The following theorem is due to~\cite{jake2}.
\begin{thm} \label{open WDVV}
For every $a,b,c\in \set{1,\ldots,m}$ the following open WDVV equations hold.
\begin{enumerate}
\item
$\d_{a} \d_{b} \d_{i} \Phi g^{ij} \d_{j} \d_{c} \Omega + \d_{a} \d_{b} \Omega \d_{c} \d_{\star} \Omega  = \d_{c} \d_{b} \d_{i} \Phi g^{ij} \d_{j} \d_{a} \Omega + \d_{c} \d_{b} \Omega \d_{a} \d_{\star} \Omega.$
\item
$\d_{a} \d_{b} \d_{i} \Phi g^{ij} \d_{j} \d_{\star} \Omega + \d_{a} \d_{b} \Omega \d^{2}_{\star} \Omega  = \d_{a} \d_{\star} \Omega \d_{b} \d_{\star} \Omega.$
\end{enumerate}
Here $\d_{\star}$ denotes differentiation with respect to $u$ and $\d_i$ denotes differentiation with respect to $w_i$.
\end{thm}

\section{Real blowups of the projective plane}
\label{blowup}
Choose a configuration $C$ of $r$ real points $x_1,\ldots,x_r,$ and $s$ pairs of conjugate points $y_1,\conj{y_1},\ldots,y_s,\conj{y_s},$  on $\C P^2$.
Let $X$ be the blowup of $\C P^2$ at the $r+2s$ points of $C$.
The complex manifold $X$ comes with a standard family of K\"ahler forms. See Chapter 1 of~\cite{GH}. For the following we take our symplectic form to be any of the standard K\"ahler forms on $X$ such that complex conjugation on $\CP^2$ lifts to an anti-symplectic involution
\begin{align*}
 \phi: X \to X, \qquad \phi^* \omega = - \omega.
\end{align*}
For the purposes of this paper, the choice does not matter because Gromov-Witten type invariants depend only on the deformation class of $\omega,$ and the family of K\"ahler forms of a fixed complex structure is convex. Deformation invariance also implies that the choice of blowup points does not matter. By extension, we often call $\phi$ conjugation. We write $L = \Fix(\phi) \subset X.$

\subsection{Homology classes of the blowup }
It will be useful to choose explicit submanifolds of $X$ that will represent the generators of $H_2(X,L;\Z)$. Let $\Line \subset X$ be the strict transform of a conjugation invariant line in $\C P^2$ not passing through any of the points of $C$. Let $\ExDiv_1, \ldots,\ExDiv_r,$ be the exceptional divisors of the $r$ real blowup points $x_1,\ldots,x_r$.
Let $\ExDiv_{r+1},\ldots,\ExDiv_{r+s},$ be the exceptional divisors of the $s$ blowup points $y_1,\ldots,y_s$, and let $\ExDiv_{r+s+1},\ldots,\ExDiv_{r+2s},$ be the exceptional divisors of the conjugate blowup points  $\conj{y_1},\ldots,\conj{y_s}$.

The Lagrangian $L$ splits $\Line$ into two hemispheres.
We denote one of these hemispheres by $H$, and give $H$ the orientation induced by the complex structure of $\Line$.
We denote by $\conj{H} = \phi(H)$ the image of $H$ under complex conjugation.
Similarly, the real submanifold $L$ splits the exceptional divisors $\ExDiv_1, \ldots, \ExDiv_r$ into two hemispheres.
We denote one of these hemispheres by $F_i$, and give it the orientation induced by the complex structure of the exceptional divisor $\ExDiv_i$.
As before, we denote the image of $F_i$ under complex conjugation by $\conj{F_i}$.

The submanifolds $H,F_i, \conj{F_i}, \ExDiv_{r+j}, \ExDiv_{r+s+j},$ represent relative homology classes
\begin{align*}
[H],[\conj{H}],[F_1],[\conj{F}_1],\ldots,[F_r],[\conj{F}_r], [\ExDiv_{r+1}],\ldots, [\ExDiv_{r+2s} ]  \in H_2(X,L;\Z).
\end{align*}
The relative homology groups $H_*(X,L)$ are given by the following lemma.
\begin{lem}
\(H_4(X,L;\Z) \simeq \Z\) with generator the image of the fundamental class under $H_4(X) \to H_4(X,L)$,
\begin{align*}
H_2(X,L;\Z) \simeq
\frac{\Z\set{[H],[\conj{H}],[F_1],[\conj{F_1}], \ldots, [F_r],[\conj{F_r}],[\ExDiv_1],\ldots,[\ExDiv_{r+2s}]}}
{[H]-[\conj{H}] = \sum_i [F_i]-[\conj{F_i}]}
\end{align*}
and \( \forall n \neq 2,4: H_n(X,L;\Z) = 0\).
\end{lem}

\subsection{The involution invariant homology groups of the blow\-up}
The conjugation $\phi:X\to X$ induces a conjugation map on the relative homology
$\phi_*: H_2(X,L;\Z) \to H_2(X,L;\Z)$.
\begin{lem}
$\phi_*$ acts on the generators of  $H_2(X,L;\Z)$ by
\begin{align*}
\phi_*: & [F_i] \mapsto [\conj{F}_i] \\
\phi_*: & [\ExDiv_{r+j}] \mapsto [\conj{\ExDiv_{r+j}}] = -[\ExDiv_{r+s+j}], \quad & j=1,\ldots, s, \\
\phi_*: & [\ExDiv_{r+s+j}] \mapsto [\conj{\ExDiv_{r+s+j}}] = -[\ExDiv_{r+j}], \quad & j=1,\ldots, s, \\
\phi_*: & [H] \mapsto [\conj{H}] = [H] - \sum_i ([F_i]-[\conj{F}_i]).
\end{align*}
\end{lem}

We define generators of $H^\phi_2(X,L;\Z)$ by projecting the generators of $H_2(X,L;\Z)$ to $H^{\phi}_{2}(X;\Q):$
\begin{align*}
\tilde{H} & = \half([H]-[\conj{H}]),  \\
\tilde{F}_i & = \half([F_i]-[\conj{F}_i]), \quad & i=1,\ldots,r \\
\tilde{E}_j & = \half([\ExDiv_{r+j}]-[\conj{\ExDiv}_{r+j}])= \half([\ExDiv_{r+j}]+[\ExDiv_{r+s+j}]), \quad & j=1, \ldots,s, \\
G_j & = \half([\ExDiv_{r+j}]+[\conj{\ExDiv}_{r+j}])= \half([\ExDiv_{r+j}]-[\ExDiv_{r+s+j}]), \quad & j=1, \ldots,s.  \\
\end{align*}
The relative homology classes $\tilde{H}, \tilde{F}_i, \tilde{E}_j,$ freely generate the relative $\phi$-invariant homology group
\begin{align*}
H^\phi_2(X,L;\Z) \simeq \Z\set{\tilde{H}, \tilde{F}_1,\ldots, \tilde{F}_r, \tilde{E}_1,\ldots, \tilde{E}_1 }.
\end{align*}
They extend to a basis of $H_2(X,L;\Q)$ by adding $G_j$.

We use the following shorthand notation for elements of the homology groups of $X$:
\begin{itemize}
\item For multi-indices
\begin{align*}
a = (a_1,\ldots,a_r),\quad  b=(b_1,\ldots,b_s), \quad c= (c_1,\ldots,c_s), \qquad a_i, b_j, c_j \in \Z
\end{align*}
 and $d \in \Z$ we denote by
$(d,a,b,c)$ the homology class
\begin{align*}
(d,a,b,c) = d [\Line] - \sum_i a_i [\ExDiv_i]  - \sum_j b_j [\ExDiv_{r+j}] - \sum_j c_j [\ExDiv_{r+s+j}]\in H_2(X;\Z).
\end{align*}
\item For multi-indices
\[
\al=(\al_1,\ldots,\al_r),\quad \beta=(\beta_1,\ldots,\beta_s), \qquad \al_i, \beta_j  \in \Z,
\]
and $d\in \Z,$ we denote by $[d,\al,\beta]$ the $\phi$-invariant relative homology class
\begin{align*}
[d,\al,\beta] = d \tilde{H} - \sum_i \al_i \tilde{F_i}  - \sum_j \beta_j \tilde{E_j} \in H^\phi_2(X,L;\Z).
\end{align*}
\end{itemize}

We use Poincar\'e duality to define a basis for the cohomology groups $X$, and to the $\phi$-invariant cohomology groups.
Let $m=\dim H^*(X;\Q)=2+r+2s$.
Define a basis for $H^*(X;\Q)$ by taking
\begin{itemize}
\item $t_0$ Poincar\'e dual to $[X]$,
\item $t_1$ Poincar\'e dual to $[\Line]$,
\item $t_{1+i}$ Poincar\'e dual to the exceptional divisor $[\ExDiv_i]$ for  $i=1\ldots,r+2s$,
\item and $t_m$ Poincar\'e dual to a point.
\end{itemize}
Define a basis for  $H^*_{\phi}(X)$ by
\begin{itemize}
\item $\tau_0  = t_0$ Poincar\'e dual to the fundamental class,
\item $\tau_1 = \half(t_1-\conj{t_1}) = t_1$,
\item $\tau_{1+i} = \half(t_{1+i}-\conj{t_{1+i}}) = t_{1+i}$ for $ i=1,\ldots,r$,
\item $\tau_{1+r+j} = \half(t_{1+r+j}-\conj{t_{1+r+j}}) =\half(t_{1+r+j}+t_{1+r+s+j})$ for $j=1,\ldots,s$,
\item $\tau_{1+r+s+j} = \half(t_{1+r+j}+\conj{t_{1+r+j}}) =\half(t_{1+r+j}+t_{1+r+s+j})$ for $j=1,\ldots,s$,
\item and $\tau_{m} = t_{m}$ Poincar\'e dual to a point.
\end{itemize}
The $\phi$-invariant cohomology $H^*_\phi(X)$ is spanned by $ \tau_0, \ldots, \tau_{1+r+s},\tau_{m} \in H^*_\phi(X)$.

\subsection{Chern numbers and Maslov index}

\begin{lem}
The Chern numbers of the generators of $H_2(X;\Z)$ are
\begin{align*}
c_1(\Line)  = 3, \quad c_1(\ExDiv_i)  = 1,\quad i=1, \ldots, r+2s.
\end{align*}
\end{lem}
\begin{proof}
Use the adjunction formula $c_1([A])= [A]\cdot [A] + \chi(\Sigma) $ for $A: \Sigma \to X$ an embedding, where $\chi$ is the Euler characteristic.
For $A:S^2 \to X$ with image $\Line$ we get $c_1(\Line) = 3$ and for  $A:S^2 \to X$ with image $\ExDiv_i$ we get $c_1(\Line) = 1$.
\end{proof}

\begin{lem}
The Maslov indices of the generators of $H_2(X,L;\Z)$ are
\begin{align*}
\mu(\tilde{H})  = 3, \qquad \mu(\tilde{F}_i)  = 1, \quad i=1, \ldots, r, \qquad \mu(\tilde{E}_j)  = 2, \quad j=1,\ldots,s.
\end{align*}
\end{lem}
\begin{proof}
Since
\begin{align*}
\mu(2\tilde{H}) & = \mu([H]-[\conj{H}]) = \mu([\Line]) = 2 c_1([\Line])= 6,
\end{align*}
we get $\mu(\tilde{H})  = 3 $.
The other Maslov indices are calculated in the same way.
\end{proof}

For a multi-index $q=(q_1,\ldots,q_k)$ we denote the sum $q_1+\cdots +q_k$ by $\abs{q}$.
The first Chern number and Maslov index are given by
\begin{align*}
c_1(d,a,b,c) & = 3d -\abs{a} -\abs{b} -\abs{c}, \\
\mu([d,\al,\beta]) & = 3d -\abs{\al} -2\abs{\beta}.
\end{align*}

\section{Gromov-Witten potentials}
In this section we use the axioms to reduce the computation of the open Gromov-Witten invariants from the computation of the multi-linear maps
\[
OGW_{[d,\al,\beta],k} : H^*_\phi(X;\Q)^{\otimes l}\to \Q
\]
to the computation of the values
\[
\Gamma_{[d,\al,\beta],k}:=OGW_{[d,\al,\beta],k}\left( \tau_{m}^{\otimes \left( \frac{3d-\abs{\al}-2\abs{\beta}-k-1}{2} \right) } \right).
\]
Specifically, we derive a formula for the open Gromov-Witten potential in terms of $\Gamma_{[d,\al,\beta],k}$. This will be useful in conjunction with Theorem~\ref{open WDVV}.
We also recall the analogous formula for the closed Gromov-Witten potential, derived in~\cite{GP}, which again will be useful in conjunction with Theorem~\ref{open WDVV}.

\subsection{The open Gromov-Witten potential}
The open Gromov-Witten potential is the formal power series in variables $T,$
\[
w=(w_{0},\ldots,w_{1+r+s},w_{m})
\]
and $u$ given by
\begin{align} \label{OpenPotetialDefinition}
\Omega(w,u) = \sum_{[d,\al,\beta],k,l} \varepsilon([d,\al,\beta]) \frac{T^{\int_{[d,\al,\beta]}\omega}u^{k}}{k! l!} OGW_{[d,\al,\beta],k}(\delta_{w}^{\otimes l}).
\end{align}
Here $\delta_{w}$ is a formal cohomology class
\begin{align*}
\delta_{w} = w_{0} \tau_{0} + \cdots + w_{1+r+s} \tau_{1+r+s} + w_{m} \tau_{m},
\end{align*}
and $\varepsilon$ is given by
\begin{align*}
\varepsilon([d,\al,\beta]) = \begin{cases} +1 & \text{ if } \mu([d,\al,\beta]) = 0 \mod 2,  \\
\sqrt{-1} & \text{ if } \mu([d,\al,\beta]) =1 \mod 2. \end{cases}
\end{align*}
Using the open Gromov-Witten axioms, we show that
\begin{align} \label{OpenPotential}
\Omega(w,u) &= w_{0} \cdot u + \\
& +\sum_{[d,\al,\beta],k} \varepsilon([d,\al,\beta]) T^{\int_{[d,\alpha,\beta]}\omega}\frac{u^{k}}{k!}  \cdot \frac{w_{m}^{l_{U}}}{l_{U}!}
\cdot e^{ \frac{d}{2} w_{1} -\sum \limits_i \frac{\al_{i}}{2}  w_{1+i} -\sum \limits_j \frac{\beta_{j}}{2}  w_{1+r+i} } \cdot \Gamma_{[d,\al,\beta],k},\notag
\end{align}
where \( l_{U}= \frac{3d-\abs{\al}-2\abs{\beta} -k -1}{2} \).
In our derivation, we separate the summands of \eqref{OpenPotetialDefinition} that depend on $w_0$, which we call the classical part, from the rest of the summands, which we call the quantum part.

Using the multi-linearity of the open Gromov-Witten invariants, the symmetry axiom and the divisor axiom, we see that the quantum part of $OGW_{[d,\al,\beta],k}(\delta_{w}^{\otimes l})$ is given by
\begin{align*}
&	\sum_{l_{1}+\cdots +l_{1+r+s} +l_{m}=l} \binom{l}{l_{1},\ldots,l_{1+r+s},l_{m} } w_{1}^{l_{1}} \cdots w_{m}^{l_{m}} \cdot OGW_{[d,\al,\beta],k}(\tau_{1}^{\otimes l_{1}} \otimes \cdots \otimes \tau_{m}^{\otimes l_{m}}) = \\
& \qquad \qquad = \sum_{l_{1}+\cdots +l_{1+r+s} +l_{m}=l}  \frac{l!}{l_{1}! \cdots l_{m}!} w_{1}^{l_{1}} \cdots w_{m}^{l_{m}} \cdot
\left(\frac{d}{2}\right)^{l_{1}} \cdot \prod_i
\left(-\frac{\al_{i}}{2}\right)^{l_{1+i}} \times \\
& \qquad \qquad\qquad \qquad \qquad \qquad \times \prod_j \left(-\frac{\beta_{j}}{2}\right)^{l_{1+r+j}} \cdot OGW_{[d,\al,\beta],k}( \tau_{m}^{\otimes l_{m}}) \\
&\qquad \qquad = l!\times\sum_{l_{1}+\cdots +l_{1+r+s} +l_{m}=l}
 \frac{(\frac{d}{2} w_{1})^{l_{1}}}{l_{1}!}  \prod_i \frac{(-\frac{\al_{i}}{2} w_{1+i})^{l_{1+i}}}{l_{1+i}!} \times \\
& \qquad\qquad\qquad \qquad \qquad \qquad \times \prod_j \frac{(-\frac{\beta_{j}}{2} w_{1+r+j})^{l_{1+r+j}}}{l_{1+r+j}!}  \cdot \frac{w_{m}^{l_{m}}}{l_{m}!} \cdot OGW_{(d,\al),k}( \tau_{m}^{\otimes l_{m}}).
\end{align*}
By the grading axiom, $OGW_{[d,\al,\beta],k}( \tau_{m}^{\otimes l})$ vanishes unless
\begin{gather*}
l \cdot \abs{\tau_{m}} + k \cdot \dim_{\R}(L) = \mu([d,\al,\beta]) + \dim_{\C}(X) +2l +k  -3 \\
\Leftrightarrow \quad  l = \frac{3d-\abs{\al}-2\abs{\beta}-k-1}{2}.
\end{gather*}
Hence setting $l_{U}= \frac{3d-\abs{\al}-2\abs{\beta} -k -1}{2},$ the quantum part of the open potential $\Omega(w,u)$ is given by
\begin{align*}
& \sum_{\substack{[d,\al,\beta],k \\ l_{1},\ldots l_{1+r+s},l_{m}}} \varepsilon([d,\al,\beta]) \frac{T^{\int_{[d,\alpha,\beta]}\omega}u^{k}}{k!}  \cdot
 \frac{(\frac{d}{2} w_{1})^{l_{1}}}{l_{1}!}  \prod_i \frac{(-\frac{\al_{i}}{2} w_{1+i})^{l_{1+i}}}{l_{1+i}!} \times \\
& \quad\qquad \qquad \qquad \qquad \times \prod_j \frac{(-\frac{\beta_{j}}{2} w_{1+r+j})^{l_{1+r+j}}}{l_{1+r+j}!}  \cdot \frac{w_{m}^{l_{m}}}{l_{m}!} \cdot OGW_{(d,\al),k}( \tau_{m}^{\otimes l_{m}}) =\\
& \quad\quad = \sum_{[d,\al,\beta],k } \sum_{l_{m}} \varepsilon([d,\al,\beta]) T^{\int_{[d,\alpha,\beta]}\omega}\frac{u^{k}}{k!}  \cdot \frac{w_{m}^{l_{m}}}{l_{m}!}
\times  \\
&\quad \qquad \qquad \qquad \qquad \qquad \times e^{ \frac{d}{2} w_{1} -\sum_i \frac{\al_{i}}{2} w_{1+i} - \sum_j \frac{\beta_{j}}{2} w_{1+r+j}  }
\cdot OGW_{[d,\al,\beta],k}( \tau_{m}^{\otimes l_{m}}) \\
& \quad\quad = \sum_{[d,\al,\beta],k } \varepsilon([d,\al,\beta]) T^{\int_{[d,\alpha,\beta]}\omega} \frac{u^{k}}{k!}  \cdot \frac{w_{m}^{l_{U}}}{l_{U}!} \cdot
e^{ \frac{d}{2} w_{1} -\sum_i \frac{\al_{i}}{2} w_{1+i} - \sum_j \frac{\beta_{j}}{2} w_{1+r+j}  }
\cdot \Gamma_{[d,\al,\beta],k}.
\end{align*}


We turn to the classical part. By the fundamental class axiom there is only one non-vanishing open Gromov-Witten invariant involving $\tau_0,$ namely \[
OGW_{[0,0,0],1}(\tau_{0}) = 1.
\]
So the classical part is
\begin{align*}
\varepsilon([0,0,0]) \frac{T^0u^{1}}{1!0!} OGW_{[0,0,0],1}(w_{0} \tau_{0}) = (+1) \cdot u \cdot w_{0} \cdot OGW_{[0,0,0],1}( \tau_{0}) = w_{0} u.
\end{align*}
Combining the classical and quantum part we obtain expression~\eqref{OpenPotential}.

\subsection{The closed Gromov-Witten potential}

The closed Gromov-Witten potential is the formal power series
\begin{align*}
\Phi(w) = \sum_{(d,a,b,c),n} \frac{T^{\int_{(d,a,b,c)}\omega}}{n!} GW_{(d,a,b,c)}(\delta_{w}^{\otimes n})
\end{align*}
in the formal variables $T$ and \( w = (w_{0},\ldots,w_{m}) \).
We combine these formal variables to a formal cohomology class \( \delta_{w} = w_{0} t_{0} + \cdots + w_{m} t_{m} \).

Using the closed Gromov-Witten axioms, a formula for the closed Gromov-Witten potential can be derived like the formula for the open Gromov-Witten potential. Namely,
\begin{align}\label{eq:cgwp}
\Phi(w) =&  \half \left\{  w_{0}^{2} w_{m}  + w_{0}w_{1}^{2} -\sum_{i=1}^{r+2s} w_{0}w_{1+i}^{2}  \right\} + \sum_{(d,a,b,c)} T^{\int_{(d,a,b,c)}\omega} \frac{w_{m}^{l_{F}}}{l_{F}!} \times\\
& \qquad \qquad \qquad \times e^{ d w_{1} -\sum_{i=1}^{r} a_{i} w_{1+i}  -\sum_{j=1}^{s} b_{j} w_{1+r+j}  -\sum_{j=1}^{s} c_{j} w_{1+r+s+j} }
\cdot  N_{(d,a,b,c)},\notag
\end{align}
where $ l_{F} = 3d -\abs{a} -\abs{b}-\abs{c} -1$, and
\begin{align*}
N_{(d,a,b,c)} = GW_{(d,a,b,c)}(t_{m}^{\otimes l_F}).
\end{align*}
See~\cite{GP} for more details.

\section{\texorpdfstring{Empty moduli spaces and vanishing open \\ Gromov-Witten invariants}{Empty moduli spaces and vanishing open Gromov-Witten invariants} }
\label{empty moduli}
In this section we show that many of the open Gromov-Witten invariants of $X = \C P^2_{r,s}$ vanish.
More specifically, we show that for many homology classes $[d,\al,\beta],$ the moduli space $\M^D_{k,l}(X/L,[d,\al,\beta])$ is empty for generic $J$ by the positivity of intersections of $J$-holomorphic curves. See~\cite{MS} Theorem E.1.5 for reference.
Also, we use the open grading axiom to show that for $d=0,1,$ many of the open invariants \(\Gamma_{[d,\al,\beta],k} \) vanish. We denote by $[q]$ the multi-index with $1$ for the $q^{th}$ entry and $0$ everywhere else.

\begin{lem} \label{lem:positive intersection consequences}
For generic $J,$ the moduli space  \( \M^D_{k,l}(X/L,[d,\al,\beta]) \) is empty if one or more of the following conditions hold:
\begin{enumerate}
\item \( d<0 \),
\item \( \al_q > d \) for some \( q=1,\ldots,r \),
\item \( 2\beta_q > d \) for some \(  q=1,\ldots, s \) and \( [d,\al,\beta] \neq [1,0,[q]] \),
\item \( \al_q < 0 \) for some \( q = 1,\ldots,r \)  and \( [d,\al,\beta] \neq [0,-[q],0] \),
\item \( \beta_q < 0 \) for some \( q = 1,\ldots,s \).
\end{enumerate}
\end{lem}

\begin{proof}
Assume \( \M^D_{k,l}(X/L,[d,\al,\beta]) \neq \emptyset\).
Then there exists a $J$-holomorphic disc \( f:(D^2, \partial D^2) \to (X,L) \) with the projection of \(f_*([D^2, \partial D^2]) \in H_2(X,L) \) to \( H^{\phi}_2(X,L) \) equal to $[d,\al,\beta]$.
Doubling $f$ we get a $J$-holomorphic curve \( u: \C P^1 \to X \) of degree
\begin{align*}
[u] = d\cdot \Line -\sum \al_i \ExDiv_i - \sum \beta_j (\ExDiv_{r+j}+\ExDiv_{r+s+j}) \in H_2(X).
\end{align*}

For each case of the lemma, we show there exists a $J$-holomorphic curve $c$ with \[
[c] \cdot [u] <0
\]
and \( [c] \neq [u] \).
By~\cite{MS} Lemma 2.4.3, the curves $c,u,$ do not agree on any open set.
This contradicts the positivity of intersections, so \( \M^D_{k,l}(X/L,[d,\al,\beta]) = \emptyset\).

Denote by $W_{\zeta,l}$ the Welschinger invariant counting real rational $J$-holo\-morphic curves of degree $\zeta$ on $X$ passing through $l$ pairs of conjugate points and $k=c_1(\zeta) -2l -1$ real points. See~\cite{Welschinger}. To treat a given case of the lemma, it suffices to construct a homology class $\zeta$ such that
\[
W_{\zeta,0} \neq 0, \qquad \zeta \cdot [u] < 0, \qquad [\zeta] \neq [u].
\]
We apply this strategy for the first three cases of the lemma.
\begin{enumerate}
\item Take $\zeta=[\Line]$.
Then $W_{\zeta,0} = 1$ and
\begin{align*}
\zeta \cdot [u] &= [\Line] \cdot \left( d[\Line] -\sum_i \al_i [\ExDiv_i] -\sum_j \beta_j (\ExDiv_{r+j}+\ExDiv_{r+s+j}) \right) \\
&= d [\Line] \cdot [\Line] =d <0.
\end{align*}

\item Take $\zeta = [\Line] - [\ExDiv_q].$
Then $W_{\zeta,0} = 1$ and
\begin{align*}
 \zeta\cdot [u] &= \left( [\Line] - [\ExDiv_q] \right) \cdot \left( d[\Line] -\sum_i \al_i [\ExDiv_i] -\sum_j \beta_j (\ExDiv_{r+j}+\ExDiv_{r+s+j}) \right)  \\
& = d -\al_q <0.
\end{align*}
\item Take $\zeta = [\Line] - [\ExDiv_{r+q}]-[\ExDiv_{r+s+q}].$
Then $W_{\zeta,0} = 1$ and
\begin{align*}
 \zeta\cdot [u] & = \left( [\Line] - [\ExDiv_{r+q}]-[\ExDiv_{r+s+q}] \right) \cdot \\
 & \qquad \cdot \left( d[\Line] -\sum_i \al_i [\ExDiv_i] -\sum_j \beta_j (\ExDiv_{r+j}+\ExDiv_{r+s+j}) \right) \\
                   & = d -2\beta_q <0.
\end{align*}
\end{enumerate}
For the remaining cases, we use the fact that if $\zeta$ is the homology class of an exceptional divisor, it admits a possibly reducible stable $J$-holomorphic representative for any $\omega$-tame $J.$ See Example 7.1.15 in~\cite{MS}. By taking $\omega$ such that $\int_\zeta \omega$ is the minimal positive value of $\omega$ on $H_2(X,\Z),$ we can make sure the $J$-holomorphic representative is irreducible. So, it suffices to find $\zeta$ the homology class of an exceptional divisor with $\zeta \cdot [u] < 0$ and $\zeta \neq [u].$
\begin{enumerate}[resume]
\item Take $\zeta=[\ExDiv_q]$.
Then
\begin{align*}
  \zeta \cdot [u] & = [\ExDiv_q] \cdot \left( d[\Line] -\sum_i \al_i [\ExDiv_i] -\sum_j \beta_j (\ExDiv_{r+j}+\ExDiv_{r+s+j}) \right)  \\
                    & = -\al_q [\ExDiv_q] \cdot [\ExDiv_q] = \al_q <0.
\end{align*}
\item Take $\zeta=[\ExDiv_{r+q}]$.
Then
\begin{align*}
  \zeta\cdot [u] & = [\ExDiv_{r+q}] \cdot \left( d[\Line] -\sum_i \al_i [\ExDiv_i] -\sum_j \beta_j (\ExDiv_{r+j}+\ExDiv_{r+s+j}) \right)  \\
                    & = -\beta_q [\ExDiv_{r+q}] \cdot [\ExDiv_{r+q}] = \beta_q <0.
\end{align*}
\qedhere
\end{enumerate}
\end{proof}

By the definition of $\Gamma_{[d,\al,\beta],k} = OGW_{[d,\al,\beta],k}(\tau_m^{\otimes l})$, we have proved the following corollary.
\begin{cor} \label{cor:positive intersection consequences}
If $\Gamma_{[d,\al,\beta],k} \neq 0 $ then $d\geq 0$ and $ 0 \leq \al_i,\beta_j \leq d$ for all $i=1,\ldots,r,$ and $j=1,\ldots,s$, except, perhaps,  for $[d,\al,\beta]=[0,-[i],0]$.
\end{cor}

We now show that for $d=0,1$ there are only few possible values of \( \al,\beta, k \) for which the open Gromov-Witten invariants do not vanish.
\begin{lem} \label{lem:OGW_of_deg=1}
If \( \Gamma_{[1,\al,\beta],k} \neq 0 \) then one of the following holds:
\begin{enumerate}
\item \( \al =0, \beta=0, k=2. \)
\item \( \al=0, \beta=0, k=0. \)
\item \( \al=[i], \beta=0, k=1, \) for some \( i = 1,\ldots,r \).
\item \( \al=[i]+[j], \beta=0, k=0, \) for some \( i\neq j \).
\item \( \al=0, \beta=[j], \) for some \( j=1,\ldots, s \).
\end{enumerate}
\end{lem}

\begin{proof}
Since \( \Gamma_{[1,\al,\beta],k} \) does not vanish, by Corollary~\ref{cor:positive intersection consequences} we have \( \al_i = 0,1, \beta_j =0,1, \) for all $i,j$.
By the grading axiom we have
\[
l = \frac{3d -\abs{\al}-2\abs{\beta}-k-1}{2} \Leftrightarrow 2l+k = 2 - \abs{\al} -2\abs{\beta}.
\]
Since  $k,l,\abs{\al},\abs{\beta}$, are all positive integers, one of the specified conditions must hold.
\end{proof}

\begin{lem} \label{lem:OGW_of_deg=0}
If \( \Gamma_{[0,\al,\beta],k} \neq 0, \) then \( \al=-[i], \beta=0, k=0, \) for some $i=1,\ldots,r$.
\end{lem}
\begin{proof}
Since \( \Gamma_{[0,\al,\beta],k} \neq 0,\) the corresponding moduli space is non-empty,  hence \( \al_i, \beta_j \leq 0 \) for all \(i,j \) by Lemma~\ref{lem:positive intersection consequences}.
In addition,
\[
l = \frac{3d -\abs{\al}-2\abs{\beta}-k-1}{2} \Leftrightarrow  2l+k= - \abs{\al} -2\abs{\beta}-1
\]
by the grading axiom. Therefore,  \(\abs{\al} \) must be negative. But Lemma~\ref{lem:positive intersection consequences} implies this is only possible if \([d,\al,\beta] = [0,-[i],0] \).
Turning our attention back to the grading axiom, we have \( 2l+k=0\) for $k,l,$ non negative integers. So \( k,l=0 \).
\end{proof}

\section{Application of open WDVV equations}
\label{WDVV equations}
We now apply the open WDVV equations to formulae~\eqref{OpenPotential} and~\eqref{eq:cgwp} for the Gromov-Witten potentials to derive explicit relations between the open and closed Gromov-Witten invariants.

As the WDVV equations use the matrix \( (g^{ij}) \), we note that for $X=\CP^2_{r,s}$ the inverse intersection matrix of the invariant homology basis
\begin{align*}
 [pt], \tilde{H}, \tilde{F}_1, \ldots, \tilde{F}_r, \tilde{E}_1, \ldots, \tilde{E}_s, [X],
 \end{align*}
 is given by
\begin{align*}
 (g^{ij}) =
 \begin{pmatrix}
0 & & &  &    &    &  &    & 1 \\
   & 1 & 0 & \cdots & 0 &    \\
   & 0 & -1 &          & 0 &    \\
   & \vdots &   & \ddots & \vdots \\
   & 0 & 0 & \cdots & -1 & \\
   & & & & & -2 &    \cdots      & 0 &    \\
   & & & & &  \vdots & \ddots & \vdots \\
   & & & & & 0 & \cdots & -2 & \\
1 & & & &  &  &  &  &0 \\
\end{pmatrix}.
\end{align*}


 \subsection{Results of open WDVV}
 In this section we apply Theorem~\ref{open WDVV} with specific choices of $a,b,c$, and extract relations between $N_{(d,\al,\beta,\gamma)}, \Gamma_{[d,\al,\beta],k},$ by examining the coefficients of the relevant power series.

In the following equations we use several shorthand notations. Given $[d,\al,\beta],k,$ we write $l = \frac{1}{2}(\mu([d,\al,\beta]) - k - 1).$ We denote by $\vdash[d,\al,\beta],k$ the the set of all integers \( d',k',l', d'', k'',l'', \) and integral multi-indices
 \begin{align*}
 \al'=(\al'_{1},\ldots, \al'_{r}), \quad \al''=(\al''_{1}, \ldots, \al''_{r}), \quad \beta'=(\beta'_{1},\ldots, \beta'_{s}), \quad \beta''=(\beta''_{1}, \ldots, \beta''_{s}),
\end{align*}  satisfying
\begin{align*}
& d'+d''=d, \quad \al'+\al''= \al, \quad \beta'+\beta''=\beta, \quad 0 \leq d',d'',k',k'',l',l'', \\
& -1 \leq \al'_{i} \leq d',\quad  -1 \leq \al''_{i} \leq d'', \quad 0 \leq 2\beta'_{j} \leq d'+1,\quad  0 \leq 2\beta''_{j} \leq d''+1, \\
&  l' = \half(\mu([d',\al',\beta'])-k'-1), \quad l''=\half(\mu([d'',\al'',\beta''])-k''-1).
\end{align*}
We take multinomial coefficients to be zero if there are non-integer or negative indices.

For $d\in \Z$, and multi-indices
\[
\al=(\al_1,\ldots,\al_r), \quad \al_i \in \Z, \qquad \beta=(\beta_1,\ldots,\beta_s),  \quad \beta_{j} \in \half \Z,
\]
we define
\begin{align*}
\tilde{N}_{[d,\al,\beta]} = \sum_{c} N_{(d,\al,\beta+c,\beta-c)},
\end{align*}
where we sum over all half integral multi-indices \( c=(c_1,\ldots, c_s), c_j \in \half \Z \).
Note that this sum is finite because $N_{(d,\al,\beta+c,\beta-c)}$ vanishes unless
\[
-1 \leq \beta_j+c_j,\beta_j-c_j \leq d,\qquad \forall j=1,\ldots, s, \]
by~\cite{GP}. We extend the definition of $\tilde N_{[d,\alpha,\beta]}$ to the case that $d$ and $\alpha$ are not integral by setting it equal to zero.

Similarly, the notation $\models[d,\al,\beta]$ denotes the set of integers \( d_{F}, l_{F}, d_{U}, l_{U} \), integral multi-indices
\begin{align*}
\al_{F}=((\al_{F})_{1}\ldots (\al_{F})_{r}), \quad  \al_{U}=((\al_{U})_{1}\ldots (\al_{U})_{r}), \quad \beta_{U}=((\beta_{U})_{1}\ldots (\beta_{U})_{s})
\end{align*}
and half integral multi-indices
\( \beta_{F}=((\beta_{F})_{1}\ldots (\beta_{F})_{s}) \)
satisfying
\begin{gather*} d_{F}+ \frac{d_{U}}{2} = \frac{d}{2},\qquad \al_{F}+ \frac{\al_{U}}{2} = \frac{\al}{2}, \qquad \beta_{F}+ \frac{\beta_{U}}{2} = \frac{\beta}{2}, \\
d_{F}, l_{F}, d_{U}, l_{U} \geq 0, \\
l_{F}=c_{1}(d_{F},\al_{F}, \beta_{F}, 0)-1, \qquad l_{U} = \half(\mu([d_{U},\al_{U},\beta_{U}]) -k-1).
\end{gather*}
Finally, we write
\begin{align*}
\al_{F} \bullet \al_{U} = \sum_{i=1}^{r} (\al_{F})_{i} \cdot (\al_{U})_{i}, \quad \beta_{F} \bullet \beta_{U} = \sum_{j=1}^{s} (\beta_{F})_{j} \cdot (\beta_{U})_{j}.
\end{align*}

\subsubsection*{Theorem \ref{open WDVV}(1) with a=b=1, c=m}
Applying Theorem \ref{open WDVV}(1) with $a=b=1, c=m,$ for $l\geq 2$ we get
\begin{align*} \label{OGW1} \tag{OGW1}
\Gamma_{[d,\al,\beta],k}
&  = \sum_{\vdash[d,\al,\beta],k} \frac{\varepsilon([d',\al',\beta']) \varepsilon([d'',\al'',\beta''])}{\varepsilon([d,\al,\beta])} \binom{k}{k', k''-1} \times \\
& \qquad \qquad \times                          \frac{d'}{2} \left[ \binom{l-2}{l'-1, l''} \frac{d''}{2} - \binom{l-2}{l',l''-1} \frac{d'}{2} \right]
										\Gamma_{[d',\al',\beta'],k'} \Gamma_{[d'',\al'',\beta''],k''} +\\
& \quad + \sum_{\models[d,\al,\beta]} \frac{\varepsilon([d_{U},\al_{U},\beta_{U}])}{\varepsilon([d,\al,\beta])}
										d_{F} \left( d_{F} \frac{d_{U}}{2} - \half \al_{F} \bullet \al_{U} -\beta_{F} \bullet \beta_{U} \right) \times \\
& \qquad \qquad  \quad \times \left[ \binom{l-2}{ l_{F}-1, l_{U}} \frac{d_{U}}{2} - \binom{l-2}{l_{F},l_{U}-1} d_{F} \right]
									               \tilde{N}_{[d_{F}, \al_{F},\beta_{F}]} \Gamma_{[d_{U},\al_{U},\beta_{U}],k}.
\end{align*}

\subsubsection*{Theorem \ref{open WDVV}(2) with a=b=1}
Applying Theorem \ref{open WDVV}(2) with $a=b=1$, for $l\geq 1, k \geq 1,$ we get
\begin{align*} \label{OGW2} \tag{OGW2}
\Gamma_{[d,\al,\beta],k}
& = \sum_{\vdash[d,\al,\beta],k} \frac{\varepsilon([d',\al',\beta']) \varepsilon([d'',\al'',\beta''])}{\varepsilon([d,\al,\beta])} \binom{l-1}{l',l''} \times \\
& \quad \quad \times  \frac{d'}{2}  \left[ \binom{k-1}{k'-1, k''-1} \frac{d''}{2} - \binom{k-1}{k',k''-2} \frac{d'}{2} \right]
										\Gamma_{[d',\al',\beta'],k'} \Gamma_{[d'',\al'',\beta''],k''}  - \\
& - \sum_{\models[d,\al,\beta]} \frac{\varepsilon([d_{U},\al_{U},\beta_{U}])}{\varepsilon([d,\al,\beta])}
										\binom{l-1}{ l_{F}, l_{U} } \times \\
& \qquad \qquad \times  d_{F}^{2} \left( d_{F} \frac{d_{U}}{2} - \half \al_{F} \bullet \al_{U} -\beta_{F} \bullet \beta_{U} \right)
										\tilde{N}_{[d_{F}, \al_{F},\beta_{F}]} \Gamma_{[d_{U},\al_{U},\beta_{U}],k} - \\
& -\delta_{k,1} \frac{d^{2}}{4} \tilde{N}_{[\frac{d}{2},\frac{\al}{2},\frac{\beta}{2}]}.
\end{align*}

\subsubsection*{Theorem \ref{open WDVV}(1) with a=b=1, c=1+i for i=1,...,r}
Fixing $i=1,...,r,$ we apply Theorem \ref{open WDVV}(1) with $a=b=1,c=1+i$.
For $ l \geq 1$ we get
\begin{align*} \label{OGW3a} \tag{OGW3.a[i]}
 (\frac{-\al_{i}}{2}) \cdot \Gamma_{[d,\al,\beta],k}	
& = \sum_{\vdash[d,\al,\beta],k} \frac{\varepsilon([d',\al',\beta']) \varepsilon([d'',\al'',\beta''])}{\varepsilon([d,\al,\beta])}
  																			 \binom{k}{k', k''-1} \binom{l-1}{l',l''} \times \\
& \qquad \qquad \times                                    \frac{d'}{2} \left[ - \frac{ d''\al'_{i} }{4} + \frac{ d' \al''_{i} }{4} \right]
																			 \Gamma_{[d',\al',\beta'],k'} \Gamma_{[d'',\al'',\beta''],k''} + \\
& \quad + \sum_{\models[d,\al,\beta]} \frac{\varepsilon([d_{U},\al_{U},\beta_{U}])}{\varepsilon([d,\al,\beta])}
																			 \binom{l-1}{l_{F}, l_{U}} d_{F} \left[ (\al_{F})_{i} \frac{d_{U}}{2} - d_{F} \frac{(\al_{U})_{i}}{2} \right] \times \\
& \quad \quad \quad \times
																			 \left( - d_{F} \frac{d_{U}}{2} + \half \al_{F} \bullet \al_{U} + \beta_F \bullet \beta_{U} \right)
																			 \tilde{N}_{[d_{F}, \al_{F},\beta_{F}]} \Gamma_{[d_{U},\al_{U},\beta_{U}],k}.
\end{align*}
In our recursive algorithm, we apply this relation for $i$ with $\al_i \neq 0$ in order to compute $\Gamma_{[d,\al,\beta],k}$.

\subsubsection*{Theorem \ref{open WDVV}(1) with a=b=1, c=1+r+j for j=1,...,s}
Fixing $j=1,...,s$ we apply Theorem \ref{open WDVV}(1) with $a=b=1,c=1+r+j$.
For $ l \geq 1$ we get
\begin{align*} \label{OGW3b} \tag{OGW3.b[j]}
(\frac{-\beta_{j}}{2}) \cdot \Gamma_{[d,\al,\beta],k}
& =  \sum_{\vdash[d,\al,\beta],k} \frac{\varepsilon([d',\al',\beta']) \varepsilon([d'',\al'',\beta''])}{\varepsilon([d,\al,\beta])}
  																			 \binom{k}{k', k''-1} \binom{l-1}{l',l''}  \times \\
& \qquad \qquad \times                                    \frac{d'}{2} \left[ - \frac{ d''\beta'_{j} }{4} + \frac{ d' \beta''_{j} }{4} \right]
																			 \Gamma_{[d',\al',\beta'],k'} \Gamma_{[d'',\al'',\beta''],k''} + \\
& \quad + \sum_{\models[d,\al,\beta]} \frac{\varepsilon([d_{U},\al_{U},\beta_{U}])}{\varepsilon([d,\al,\beta])}
																			 \binom{l-1}{l_{F}, l_{U}} d_{F}
																			 \left[ (\beta_{F})_{j} \frac{d_{U}}{2} - d_{F} \frac{(\beta_{U})_{j}}{2} \right] \times \\
& \quad \quad \quad \times                          \left( - d_{F} \frac{d_{U}}{2} + \half \al_{F} \bullet \al_{U} + \beta_F \bullet \beta_{U} \right)
																			 \tilde{N}_{[d_{F}, \al_{F},\beta_{F}]} \Gamma_{[d_{U},\al_{U},\beta_{U}],k}.
\end{align*}
In our recursive algorithm, we apply this relation for $j$ with $\beta_j \neq 0$ in order to compute $ \Gamma_{[d,\al,\beta],k}$.

\pagebreak
\subsection{Results of combining relations (OGW1)-(OGW3b)}
\subsubsection*{Combining \eqref{OGW1} and \eqref{OGW2}}
We compute $\Gamma_{[d+1,\al,\beta],k-1}$ using \eqref{OGW1} and \eqref{OGW2}.
Subtracting \eqref{OGW1} from \eqref{OGW2}, for $k \geq 2$ we get
\begin{align*} \label{OGW4} \tag{OGW4}
&   \frac{\varepsilon([d,\al,\beta])\varepsilon([1,0,0])}{\varepsilon([d+1,\al,\beta])} \cdot \frac{d+1}{4} \cdot \Gamma_{[1,0,0],0} \Gamma_{[d,\al,\beta],k} =
\\
& = \sum_{\substack{\vdash[d+1,\al,\beta],k-1 \\ ([d',\al',\beta'],k') \neq ([d,\al,\beta],k) \\ ([d'',\al'',\beta''],k'') \neq ([d,\al,\beta],k) }}
                                       \frac{\varepsilon([d',\al',\beta']) \varepsilon([d'',\al'',\beta''])}{\varepsilon([d+1,\al,\beta])} 	 \binom{l+1}{l',l''} \times \\
& \qquad \qquad \quad \times \frac{d'}{2} \left[ \binom{k-2}{k'-1, k''-1} \frac{d''}{2} - \binom{k-2}{k',k''-2} \frac{d'}{2} \right]
        								\Gamma_{[d',\al',\beta'],k'} \Gamma_{[d'',\al'',\beta''],k''}  -
\\
& \qquad - \sum_{\models[d+1,\al,\beta]} \frac{\varepsilon([d_{U},\al_{U},\beta_{U}])}{\varepsilon([d+1,\al,\beta])}
										\binom{l+1}{ l_{F}, l_{U} } \times \\
& \qquad \qquad \qquad \times d_{F}^{2}
										\left( d_{F} \frac{d_{U}}{2} - \half \al_{F} \bullet \al_{U} -\beta_{F} \bullet \beta_{U} \right)
        							    \tilde{N}_{[d_{F}, \al_{F},\beta_{F}]} \Gamma_{[d_{U},\al_{U},\beta_{U}],k-1}  -
\\
&	\qquad  -	\delta_{k,2} \frac{(d+1)^{2}}{4} \tilde{N}_{[\frac{d+1}{2},\frac{\al}{2},\frac{\beta}{2}]} -
\\
&  -\sum_{\substack{\vdash[d+1,\al,\beta],k-1 \\ ([d',\al',\beta'],k') \neq ([d,\al,\beta],k) \\ ([d'',\al'',\beta''],k'') \neq ([d,\al,\beta],k) }}
                                       \frac{\varepsilon([d',\al',\beta']) \varepsilon([d'',\al'',\beta''])}{\varepsilon([d+1,\al,\beta])}   \binom{k-1}{k', k''-1} \times \\
& \qquad \qquad \qquad \times  										 
  										\frac{d'}{2} \left[ \binom{l}{l'-1, l''} \frac{d''}{2} - \binom{l}{l',l''-1} \frac{d'}{2} \right]
                                       \Gamma_{[d',\al',\beta'],k'} \Gamma_{[d'',\al'',\beta''],k''} -
\\
& \qquad -\sum_{\models[d+1,\al,\beta]} \frac{\varepsilon([d_{U},\al_{U},\beta_{U}])}{\varepsilon([d+1,\al,\beta])}
										d_{F} \left( d_{F} \frac{d_{U}}{2} - \half \al_{F} \bullet \al_{U} -\beta_{F} \bullet \beta_{U} \right) \times \\
&	\qquad \qquad \qquad          \times \left[ \binom{l}{ l_{F}-1, l_{U}} \frac{d_{U}}{2} - \binom{l}{l_{F},l_{U}-1} d_{F} \right]
                                       \tilde{N}_{[d_{F}, \al_{F},\beta_{F}]} \Gamma_{[d_{U},\al_{U},\beta_{U}],k-1}.
\end{align*}

\pagebreak
\subsubsection*{Combining \eqref{OGW2} and \eqref{OGW3a}}
Compute $\Gamma_{[d+1,\al,\beta],k+1}$ using \eqref{OGW2} and \eqref{OGW3a}.
Subtracting \eqref{OGW3a} from \eqref{OGW2}, for $\alpha_i \neq 0$ we get

\begin{align*}  \label{OGW5a} \tag{OGW5.a[i]}
&  \frac{\varepsilon([d,\al,\beta])\varepsilon([1,0,0])}{\varepsilon([d+1,\al,\beta])} \cdot \frac{d^2+(1-k)d-k}{4} \cdot \Gamma_{[1,0,0],2} \Gamma_{[d,\al,\beta],k} =  \\
& = \sum_{\substack{\vdash[d+1,\al,\beta],k+1 \\ ([d',\al',\beta'],k') \neq ([d,\al,\beta],k) \\ ([d'',\al'',\beta''],k'') \neq ([d,\al,\beta],k) }}
                                       \frac{\varepsilon([d',\al',\beta']) \varepsilon([d'',\al'',\beta''])}{\varepsilon([d+1,\al,\beta])} \binom{l}{l',l''} \times \\
& \qquad \qquad \quad \times \frac{d'}{2}
    									\left[ \binom{k}{k'-1, k''-1} \frac{d''}{2} - \binom{k}{k',k''-2} \frac{d'}{2} \right]
										\Gamma_{[d',\al',\beta'],k'} \Gamma_{[d'',\al'',\beta''],k''} - \\
& \qquad \quad - \sum_{\models[d+1,\al,\beta]} \frac{\varepsilon([d_{U},\al_{U},\beta_{U}])}{\varepsilon([d+1,\al,\beta])}
										\binom{l}{ l_{F}, l_{U} } \times \\
& \qquad \qquad \qquad \quad \times d_{F}^{2}
										\left( d_{F} \frac{d_{U}}{2} - \half \al_{F} \bullet \al_{U} -\beta_{F} \bullet \beta_{U} \right)
										\tilde{N}_{[d_{F}, \al_{F},\beta_{F}]} \Gamma_{[d_{U},\al_{U},\beta_{U}],k+1} - \\
&	\qquad \quad -	\delta_{k,0} \frac{(d+1)^{2}}{4} \tilde{N}_{[\frac{d+1}{2},\frac{\al}{2},\frac{\beta}{2}]} + \\
& +\left(\frac{2}{\al_{i}} \right) \cdot
      \sum_{\substack{\vdash[d+1,\al,\beta],k+1 \\ ([d',\al',\beta'],k') \neq ([d,\al,\beta],k) \\ ([d'',\al'',\beta''],k'') \neq ([d,\al,\beta],k) }}
                                       \frac{\varepsilon([d',\al',\beta']) \varepsilon([d'',\al'',\beta''])}{\varepsilon([d+1,\al,\beta])}
  										\binom{k+1}{k', k''-1} \binom{l}{l',l''} \times \\
& \qquad \qquad \qquad \qquad \quad \times \frac{d'}{2}
										\left[ - \frac{ d''\al'_{i} }{4} + \frac{ d' \al''_{i} }{4} \right]
										\Gamma_{[d',\al',\beta'],k'} \Gamma_{[d'',\al'',\beta''],k''}  \\
& +\left(\frac{2}{\al_{i}} \right) \cdot \sum_{\models[d+1,\al,\beta]} \frac{\varepsilon([d_{U},\al_{U},\beta_{U}])}{\varepsilon([d+1,\al,\beta])}
										\binom{l}{l_{F}, l_{U}} d_{F}
										\left[ (\al_{F})_{i} \frac{d_{U}}{2} - d_{F} \frac{(\al_{U})_{i}}{2} \right] \times \\
& \qquad \qquad \quad  \qquad \quad \times										 
										\left( - d_{F} \frac{d_{U}}{2} + \half \al_{F} \bullet \al_{U} + \beta_F \bullet \beta_{U} \right)
										\tilde{N}_{[d_{F}, \al_{F},\beta_{F}]} \Gamma_{[d_{U},\al_{U},\beta_{U}],k+1}.
\end{align*}

\pagebreak
\subsubsection*{Combining \eqref{OGW2} and \eqref{OGW3b}}
Compute $\Gamma_{[d+1,\al,\beta],k+1}$ using \eqref{OGW2} and \eqref{OGW3b}.
Subtracting \eqref{OGW3b} from \eqref{OGW2}, for $\beta_j \neq 0$ we get
\begin{align*} \label{OGW5b} \tag{OGW5.b[j]}
&  \frac{\varepsilon([d,\al,\beta])\varepsilon([1,0,0])}{\varepsilon([d+1,\al,\beta])} \cdot \frac{d^2+(1-k)d-k}{4} \cdot \Gamma_{[1,0,0],2} \Gamma_{[d,\al,\beta],k} =  \\
& = \sum_{\substack{\vdash[d+1,\al,\beta],k+1 \\ ([d',\al',\beta'],k') \neq ([d,\al,\beta],k) \\ ([d'',\al'',\beta''],k'') \neq ([d,\al,\beta],k) }}
                                       \frac{\varepsilon([d',\al',\beta']) \varepsilon([d'',\al'',\beta''])}{\varepsilon([d+1,\al,\beta])}
    									\binom{l}{l',l''} \times \\
& \qquad \qquad \quad \times \frac{d'}{2}
    									\left[ \binom{k}{k'-1, k''-1} \frac{d''}{2} - \binom{k}{k',k''-2} \frac{d'}{2} \right]
										\Gamma_{[d',\al',\beta'],k'} \Gamma_{[d'',\al'',\beta''],k''} - \\
& \qquad - \sum_{\models[d+1,\al,\beta]} \frac{\varepsilon([d_{U},\al_{U},\beta_{U}])}{\varepsilon([d+1,\al,\beta])}
										\binom{l}{ l_{F}, l_{U} } \times \\
& \qquad \qquad \qquad \times d_{F}^{2}
										\left( d_{F} \frac{d_{U}}{2} - \half \al_{F} \bullet \al_{U} -\beta_{F} \bullet \beta_{U} \right) -
										\tilde{N}_{[d_{F}, \al_{F},\beta_{F}]} \Gamma_{[d_{U},\al_{U},\beta_{U}],k+1} \\
&	\qquad -	\delta_{k,0} \frac{(d+1)^{2}}{4} \tilde{N}_{[\frac{d+1}{2},\frac{\al}{2},\frac{\beta}{2}]} \\
& +\left(\frac{2}{\beta_{j}} \right) \cdot
      \sum_{\substack{\vdash[d+1,\al,\beta],k+1 \\ ([d',\al',\beta'],k') \neq ([d,\al,\beta],k) \\ ([d'',\al'',\beta''],k'') \neq ([d,\al,\beta],k) }}
                                       \frac{\varepsilon([d',\al',\beta']) \varepsilon([d'',\al'',\beta''])}{\varepsilon([d+1,\al,\beta])}
  										\binom{k+1}{k', k''-1} \binom{l}{l',l''} \times \\
& \qquad \qquad \qquad \qquad \times \frac{d'}{2}
										\left[ - \frac{ d''\beta'_{j} }{4} + \frac{ d' \beta''_{j} }{4} \right]
										\Gamma_{[d',\al',\beta'],k'} \Gamma_{[d'',\al'',\beta''],k''}  \\
& +\left(\frac{2}{\beta_{j}} \right) \cdot \sum_{\models[d+1,\al,\beta]} \frac{\varepsilon([d_{U},\al_{U},\beta_{U}])}{\varepsilon([d+1,\al,\beta])}
										\binom{l}{l_{F}, l_{U}} d_{F} 			
										\left[ (\beta_{F})_{j} \frac{d_{U}}{2} - d_{F} \frac{(\beta_{U})_{j}}{2} \right] \times \\
& \qquad \qquad \qquad \qquad \times \left( - d_{F} \frac{d_{U}}{2} + \half \al_{F} \bullet \al_{U} + \beta_F \bullet \beta_{U} \right)
										\tilde{N}_{[d_{F}, \al_{F},\beta_{F}]} \Gamma_{[d_{U},\al_{U},\beta_{U}],k+1}.
\end{align*}

\section{Welschinger invariants and initial calculations of open Gromov-Witten invariants}
\label{initial calculations}

The open Gromov-Witten invariants are related to the Welschinger invariants by a simple formula.
In this section we recall this formula, due to the second author in~\cite{jake1,jake2}, and use it to compute some of the invariants $\Gamma_{[d,\al,\beta],k}$ for $d=0,1$.

We first relate every $\phi$-invariant relative homology class \( \theta \in H^\phi_2(X,L;\Q) \) to a $\phi$-invariant homology class \( \tilde{\theta} \in H^\phi_2(X;\Q) \) by imitating the doubling of $J$-holomorphic discs.
Denote by
\begin{align*}
\mathfrak{j} : H^\phi_2(X;\Q ) \to H^\phi_2(X,L;\Q)
\end{align*}
the map induced by the long exact sequence of $(X,L)$.
Since the conjugation $\phi$ acts trivially on $L,$ the map $\mathfrak{j}$ is an isomorphism.
Therefore, for every \( \theta \in H^\phi_2(X,L;\Q) \) we can define \( \tilde{\theta} \in H^\phi_2(X;\Q) \) by \( \mathfrak{j} : \tilde{\theta} \mapsto 2 \theta \).

For example, for $X=\CP^2_{r,s}$ and the relative homology class
\begin{align*}
\theta = [d,\al,\beta] = d \tilde{H} - \sum_i \al_i \tilde{F}_i -\sum_j \beta_j \tilde{E}_j  \in H^\phi_2(X,L;\Z),
\end{align*}
we have
\begin{align*}
\tilde{\theta} = (d, \al, \beta,\beta) = d \Line - \sum_i \al_i \ExDiv_i -\sum_j \beta_j \ExDiv_{r+j} - \sum_j \beta_j \ExDiv_{r+s+j} \in  H^\phi_2(X;\Z).
\end{align*}

Next, following~\cite{jake2}, there is a one-to-one correspondence between $Pin$ structures $\mathfrak{p}$ on $L,$ and functions
\( t_{\mathfrak{p}}:H_1(L; \Z/2\Z) \to \Z/2\Z\) that satisfy
\[
t_\mathfrak{p} (x+y)=t(x)+t(y)+x \cdot y + w_1(x)w_1(y)  \mod{2}.
\]
Here $w_1 \in H^1(L;\Z/2\Z)$ is the first Stiefel-Whitney class of the tangent bundle $TL$. The function
$t_{\mathfrak{p}}$ is determined by its values on the generators of $H_1(L; \Z/2\Z)$.
We use $t_{\mathfrak{p}}$ to define another function
\begin{align*}
s_\mathfrak{p} : H^\phi_2(X,L)\to \Z/2\Z, \quad s_{\mathfrak{p}}(\theta) = \frac{ \mu(\theta) - \tilde{\theta} \cdot \tilde{\theta} -2 }{2} + t_{\mathfrak{p}}(\d \theta) +1.
\end{align*}

We denote by $W_{\zeta,l}$ the Welschinger invariant counting real rational $J$-holo\-morphic curves of degree $\zeta$ on  $X$ passing through $l$ pairs of conjugate points and $k=c_1(\zeta) -2l -1$ real points. See~\cite{Welschinger}.
Then the open Gromov-Witten invariants and the Welschinger invariants are related by the formula
\begin{align} \label{ogw_welschinger_relation}
OGW_{\theta,k}(\tau_m^{\otimes l}) = (-1)^{s_\mathfrak{p}} 2^{1-l} W_{\tilde{\theta},l},
\end{align}
where $\tau_m\in H^4(X)$ is the Poincar\'e dual of a point.

We now use relation \eqref{ogw_welschinger_relation} to calculate a small set of $\Gamma_{[d,\al,\beta],k}$ values, which serve as the initial conditions for the recursive calculation of Section~\ref{proof of recursion}.
\begin{lem}\label{lm:init} There is a choice of $Pin$ structure $\p$ such that
\begin{gather*}
\Gamma_{[0,-[i],0],0}  = 2, \qquad
\Gamma_{[1,0,0],2} = 2, \qquad
\Gamma_{[1,0,0],0} =1, \\
\Gamma_{[1,[i],0],1}  = 2, \quad i=1,\ldots,r, \qquad
\Gamma_{[1,0,[j]],0}= 2, \quad j=1,\ldots,s.
\end{gather*}
\end{lem}
\begin{proof}
We choose $\p$ such that $t_{\mathfrak{p}}:H^1(L;\Z/2\Z) \to \Z/2\Z$ evaluates to $1$ on the generators $ [\d H], [\d F_1],\ldots, [\d F_r] \in H^2(L;\Z/2\Z)$.
A simple calculation shows that $ t_\mathfrak{p}( \d H + \d F_i) = 1 $. Therefore, we have
\begin{align*}
s_\mathfrak{p}([0,-[i],0]) = 0, \quad  s_\mathfrak{p}([1,0,0])=0, \quad   s_\mathfrak{p}([1,[i],0]) = 0, \quad  s_\mathfrak{p}([1,0,[j]]) = 0.
\end{align*}
Applying relation \eqref{ogw_welschinger_relation}, we have
\begin{align*}
\Gamma_{[d,\al,\beta],k} = (-1)^{s_\mathfrak{p}} 2^{1-l} W_{(d,\al,\beta,\beta),l}, \quad l= \half (3d -\abs{\al} -2\abs{\beta}-k-1).
\end{align*}
The results now follow from the known fact that $W_{(d,\alpha,\beta,\beta),l} = 1$ for each choice of $d,\alpha,\beta,l,$ considered in the lemma.
\end{proof}

For the remainder of the paper, we use the $Pin$ structure $\p$ given by Lemma~\ref{lm:init}.

\section{Proof of recursion}
\label{proof of recursion}

We will show that the $\Gamma_{[d,\al,\beta],k}$ calculated in Section \ref{initial calculations} and the relations \eqref{OGW1}-\eqref{OGW5b}, together with the closed Gromov-Witten invariants, determine all the open Gromov-Witten invariants $\Gamma_{[d,\al,\beta],k}$  of $\CP^2_{r,s}$ by induction on the integers $d$, \( \al_1,\ldots,\al_r\), \(\beta_1,\ldots,\beta_s \), and $k$.
We can perform induction on these a priori signed integers since by Corollary~\ref{cor:positive intersection consequences} we can assume that $d\geq0$ and $\al_i \geq 0$ unless $[d,\al,\beta]=(0,-[i],0)$ for some $i$, which we take as one of the initial conditions.

Formally, we define a partial order on the indices $([d,\al,\beta],k)$ where $d,k,$ are integers and $\al=(\al_1,\ldots,\al_r), \beta=(\beta_1,\ldots,\beta_s),$ are integral multi-indices.
We say that $([d',\al',\beta'],k')<([d,\al,\beta],k)$ if
\begin{itemize}
\item $d'<d$, or
\item $d'=d$, and $\al'_i \leq \al_i, \beta'_j \leq \beta_j$ for all $i=1,\ldots,r, j=1,\ldots, s,$ with at least one the inequalities strict, or
\item $d'=d, \al'=\al, \beta' = \beta,$ and $k'<k$.
\end{itemize}
We also write $0\leq ([d,\al,\beta],k)$ for $ 0 \leq d, \al_1,\ldots,\al_r, \beta_1,\ldots,\beta_s, k $.
The following lemma is immediate from the open grading axiom.
\begin{lem} \label{lem:finite indices}
Let $([d,\al,\beta],k)$ be an index as above.
The set
\begin{align*}
\set{([d',\al',\beta'],k') : 0 \leq ([d',\al',\beta'],k') < ([d,\al,\beta],k), \; \Gamma_{[d',\al',\beta'],k'} \neq 0}
\end{align*}
 is finite.
\end{lem}

The following lemmas show that\eqref{OGW1} - \eqref{OGW5b} express $\Gamma_{[d,\al,\beta],k}$ as a function of $\Gamma_{[d',\al',\beta'],k'}$ for $([d',\al',\beta'],k')<([d,\al,\beta],k)$.
\begin{lem} \label{lem:OGW123_recursive}
\begin{enumerate}
\item \eqref{OGW1} relates $\Gamma_{[d,\al,\beta],k}$ (with $l\geq 2$) only to open Gromov-Witten invariants $\Gamma_{[d',\al',\beta'],k'}$ with $d'<d$.
\item \eqref{OGW2} relates $\Gamma_{[d,\al,\beta],k}$ (with $l\geq 1, k\geq 1$) only to open Gromov-Witten invariants $\Gamma_{[d',\al',\beta'],k'}$ with $d'<d$.
\item If $\al_i\neq 0$, then \eqref{OGW3a} relates $\Gamma_{[d,\al,\beta],k}$ (with $l\geq 1$) only to open Gromov-Witten invariants $\Gamma_{[d',\al',\beta'],k'}$ with $d'<d$.
\item If $\beta_j\neq 0$, then \eqref{OGW3b} relates $\Gamma_{[d,\al,\beta],k}$ (with $l\geq 1$) only to open Gromov-Witten invariants $\Gamma_{[d',\al',\beta'],k'}$ with $d'<d$.
\end{enumerate}
\end{lem}
\begin{proof}
Examine the first sum in \eqref{OGW1}.
It vanishes for $d'=0$, so $d'+d''=d$ implies $d''<d$.
Assume that $d'=d, d''=0$. Then by Lemma \ref{lem:OGW_of_deg=0}  we must have $\Gamma_{[0,\al'',\beta''],k''}=0$ or $k''=0$.
In both cases the summand vanishes, in the first case because the open Gromov-Witten invariant vanishes, and in the second case because the binomial coefficient $\binom{k}{k',-1}=0 $ vanishes.

Examine the second sum.
Assume $d_U=d$. Then $d_F+\frac{d_U}{2} = \frac{d}{2}$ implies $d_F=0$, so the sum vanishes.
Applying the same arguments gives the analogous results for \eqref{OGW2}, \eqref{OGW3a},\eqref{OGW3b}.
\end{proof}

Next we show that applying \eqref{OGW4} reduces either $d$, one of the $\al_i$'s, one of the $\beta_j$'s, or $k$.
\begin{lem} \label{lem:OGW4_recursive}
\eqref{OGW4} relates $\Gamma_{[d,\al,\beta],k}$ (with $k\geq 2$ boundary points) only to open Gromov-Witten invariants $\Gamma_{[d',\al',\beta'],k'}$ with \[ ([d',\al',\beta'],k') < ([d,\al,\beta],k). \]
\end{lem}
\begin{proof}
We prove that $d',d'',d_U \leq d$ in all the sums involved by following the argument in Lemma \ref{lem:OGW123_recursive}. After all, \eqref{OGW4} is just the difference of \eqref{OGW1} and \eqref{OGW2} applied to $\Gamma_{[d+1,\al,\beta],k-1}$.
The sums involving closed invariants do not include summands with $d_U=d$, since $ d_F+\frac{d_U}{2} = \frac{d+1}{2}$ would imply $d_F=\half$ which is not integral. We ignore these sums for the rest of the proof, focusing on the sums not containing closed invariants.

Next we show that $\al'_i,\al''_i \leq \al_i$ for all $i$ and $\beta'_j,\beta''_j \leq \beta_j$ for all $j$.
We already know that $d',d''\leq d$ and $d'+d''=d+1$, so $d',d'' \geq 1$, and by Corollary~\ref{cor:positive intersection consequences} we have
\( \al'_i,\al''_i,\beta'_j,\beta''_j \geq 0 \)
 for all $i,j$.
Our claim now follows from
\( \al'+\al''=\al, \beta'+\beta''=\beta. \)

So far we showed that $d',d''\leq d$, that $\al'_i,\al''_i \leq \al_i$ for all $i$, and that $\beta'_j, \beta''_j \leq \beta_j$ for all $j$.
To finish the proof we have to show that no summand contains $\Gamma_{[d,\al,\beta],K}$ with $K\geq k$.
Assume $[d',\al',\beta']=[d,\al,\beta]$ (the argument for $[d'',\al'',\beta'']=[d,\al,\beta]$ is the same).
Then $[d'',\al'',\beta'']=[1,0,0]$ and by Lemma \ref{lem:OGW_of_deg=1} either $\Gamma_{[d'',\al'',\beta''],k''}$ vanishes or $k''=0,2$.
The binomial coefficients then imply that $k'+k''=k$, therefore $k'=k,k-2$.
The case $([d',\al',\beta'],k')=([d,\al,\beta],k)$ was explicitly removed from the sum, so $k'=k-2$ and we're done.
\end{proof}

\begin{lem} \label{lem:OGW5_recursive}
The relation \eqref{OGW5a} relates $\Gamma_{[d,\al,\beta],k}$ having $l=0$ and $\al_i\neq 0$ for some $i$ only to open Gromov-Witten invariants $\Gamma_{[d',\al',\beta'],k'}$ with $ ([d',\al',\beta'],k') < ([d,\al,\beta],k)$.
The relation \eqref{OGW5b} relates $\Gamma_{[d,\al,\beta],k}$ having $l=0$ and $\beta_j\neq 0$ for some $j$ only to open Gromov-Witten invariants $\Gamma_{[d',\al',\beta'],k'}$ with $ ([d',\al',\beta'],k') < ([d,\al,\beta],k)$.
\end{lem}
\begin{proof}
Showing that
\[
d',d''\leq d, \qquad d_U < d, \qquad \al'_i,\al''_i \leq \al_i, \qquad \beta'_j,\beta''_j \leq \beta_j,
\]
is the same as the last proof. It remains to prove that all summands containing $\Gamma_{[d,\al,\beta],K}$ with $K \geq k $ vanish.

For $[d',\al',\beta']=[d,\al,\beta]$ we have  $[d'',\al'',\beta'']=[1,0,0]$.
Therefore $\Gamma_{[d'',\al'',\beta''],k''}$ vanishes or $k''=0,2$.
The binomial coefficients now imply  $k'+k''=k+2$, so $k'=k+2,k$.
As the case $k'=k$ is explicitly removed from the sum, we are left with proving that $\Gamma_{[d,\al,\beta],k+2}$ vanishes.
Applying the grading axiom, we note that  $0=\half (\mu([d,\al,\beta])-k-1)$ since $\Gamma_{[d,\al,\beta],k} $ has $0$ interior points.
Now we apply the grading axiom to $\Gamma_{[d,\al,\beta],k+2}$ and we get \[
l= \half (\mu([d,\al,\beta])-(k+2)-1) =\half (\mu([d,\al,\beta])-k-1)-1 = -1,
\]
so  $\Gamma_{[d,\al,\beta],k+2} $ vanishes.

The same argument gives the analogous result for \eqref{OGW5b}.
\end{proof}

We now collect the results of the last lemmas and the calculations of $\Gamma_{[d,\al,\beta],k}$ for $d=0,1,$ to prove the main theorem.

\begin{thm} \label{recursive_OGW_relation}
The open Gromov-Witten invariants \( \Gamma_{[d,\al,\beta],k} \) are recursively determined by the relations \eqref{OGW1} - \eqref{OGW5b}, the vanishing results of Corollary~\ref{cor:positive intersection consequences}, the closed Gromov-Witten invariants \( \tilde{N}_{[d,\al,\beta]} = \sum_{c} N_{(d,\al,\beta+c,\beta-c)} \), and the following initial values:
\begin{enumerate}
\item $\Gamma_{[0,-[i],0],0}=2$ for $i=1,\ldots,r$,
\item $\Gamma_{[1,0,0],2}=2,$
\item $\Gamma_{[1,0,0],0}=1,$
\item $\Gamma_{[1,[i],0],1}=2$ for $i=1,\ldots,r$,
\item $\Gamma_{[1,0,[j]],0}=2$ for $j=1,\ldots,s$.
\end{enumerate}
\end{thm}
\begin{proof}
We prove by induction on $([d,\al,\beta],k)$.

\emph{Proof of induction base.}
All the initial conditions were verified in Lemma~\ref{lm:init}.
By Corollary \ref{cor:positive intersection consequences} and the grading axiom, $\Gamma_{[d,\al,\beta],k}=0$ for $ 0 \not \leq ([d,\al,\beta],k)$, except for the $r$ initial values $ ([d,\al,\beta],k)=([0,-[i],0],0).$

For $d=0,1,$ we know by Lemmas~\ref{lem:OGW_of_deg=0} and~\ref{lem:OGW_of_deg=1} that the invariants vanish except for the cases
\begin{align*}
\Gamma_{[0,-[i],0],0}, \Gamma_{[1,0,0],2}, \Gamma_{[1,0,0],0}, \Gamma_{[1,[i],0],1}, \Gamma_{(1,[i]+[j]),0}, \Gamma_{[1,0,[j]],0}.
\end{align*}
These invariants are all initial values except $\Gamma_{[1,[i]+[j],0],0}$, which are calculated using the initial values and equation \eqref{OGW5a}.

\emph{Proof of induction step.}
We show that for $d\geq 2$ we can always apply one of the relations \eqref{OGW1}- \eqref{OGW5b}.
By Lemmas \ref{lem:OGW123_recursive}, \ref{lem:OGW4_recursive}, and \ref{lem:OGW5_recursive}, the relations \eqref{OGW1}-\eqref{OGW5b} determine $\Gamma_{[d,\al,\beta],k}$ as a function of  $\Gamma_{[d',\al',\beta'],k'}$ for $ 0 \leq([d',\al',\beta'],k') < ([d,\al,\beta],k)$ and of $\Gamma_{[0,-[i],0],0}$.
Since the set of such invariants which do not vanish is finite by Lemma~\ref{lem:finite indices}, this completes the proof.

Let $([d,\al,\beta],k)$ be an index with $d \geq 2$.
We cover all possible cases and show that we can apply a recursive relation.
\begin{enumerate}
\item If $l= \half (\mu([d,\al,\beta])-k-1 \geq 2$ we can apply \eqref{OGW1}.
\item If $l=1$ then either $k\geq 1$ and we can apply \eqref{OGW2} or $k=0$.
In the latter case we use the grading axiom to see that $ \abs{\al}+2\abs{\beta}=3d-3 >0$, so there exists some $\al_i\neq 0$ or some $\beta_j \neq 0$, and therefore we can apply \eqref{OGW3a} or \eqref{OGW3b}.
\item For the case $l=0$ we will use \eqref{OGW4}, \eqref{OGW5a}, and \eqref{OGW5b}.
To calculate $\Gamma_{[d,\al,\beta],k}$ using these relations we need its coefficient to be non zero.
By Lemma~\ref{lm:init} we have $\Gamma_{[1,0,0],0} = 1,$ so for $k\geq 2$ we can apply \eqref{OGW4}.
Therefore, we only have to deal we the cases $k=0,1$. We do this using relations \eqref{OGW5a} and \eqref{OGW5b}.
The coefficient of $\Gamma_{[d,\al,\beta],k}$ does not vanish since $\Gamma_{[1,0,0],2} = 2$ by Lemma~\ref{lm:init}, and $d^2+(1-k)d-k =0 $ only when $k=0, d=0,-1$ or $k=1,d=\pm 1$.
By the grading axiom, we have $\abs{\al} + 2\abs{\beta} = 3d-k-1 >0$. So, there exists $\al_i \neq 0$ or $\beta_j \neq 0$, and we can apply either \eqref{OGW5a} or \eqref{OGW5b}.
\qedhere
\end{enumerate}
\end{proof}

Theorem \ref{complete OGW} is an immediate corollary.

\section{Example computations}\label{sec:example}
In this section we list some calculated values of open Gromov-Witten invariants.
The invariants were calculated by a Maple program implementing the recursive algorithm implied by Theorem \ref{recursive_OGW_relation}.

We first reiterate our notations.
For an integer $d\in \Z$ and integral multi-indices \( \al=(\al_1,\ldots,\al_r), \beta=(\beta_1,\ldots,\beta_s), \) we write \( [d,\al,\beta] \in H^\phi_2(X,L;\Z) \) for the relative homology class
\begin{align*}
[d,\al,\beta] = d \tilde{H} -\sum_i \al_i \tilde{F}_i -\sum_j \tilde{E}_j.
\end{align*}
The homology classes \(\tilde{H}, \tilde{F}_i, \tilde{E}_j \in H^\phi_2(X,L;\Z)  \) are the $\phi$-invariant projections of relative homology classes represented by $H$ a hemisphere of a real line with real boundary, $F_i$ a hemisphere of the $i^{th}$ exceptional divisor with real boundary, and $\ExDiv_{r+j}$ an imaginary exceptional divisor.

For notational convenience we write $0$ for the empty multi-index and use exponential notation for multi-indices of repetitive values, for example
\begin{align*}
[10,(2^3),(3^2)]  =[10, (2,2,2), (3,3)] = 10 \tilde{H} -2\tilde{F}_1 -2\tilde{F}_1 -2\tilde{F}_3 -3\tilde{E}_1 -3\tilde{E}_2.
\end{align*}

We denote by \( \Gamma_{[d,\al,\beta],k} = OGW_{[d,\al,\beta],k} (\tau_m^{\otimes l} ) \) the open Gromov-Witten invariant of degree
\( [d,\al,\beta] \in H^\phi_2(X,L;\Z) \) and $k=0,1,2,\ldots$ boundary points and \( l= \half( 3d -\abs{\al} -2\abs{\beta} -k-1) \) internal points.

\begin{table}[ht]
\centering
\caption{Open Gromov-Witten invariants of a $\CP^2$ blowup  at $r$ real points.}
\begin{tabular}{|l|l|l|}
\hline
$d=6, \beta=0$ & $d=7, \beta=0$ & $d=8, \beta=0$ \\
\hline
$ \Gamma_{[6, (2^5),0], 7} = 8320 $ & $ \Gamma_{[7, (2^5),0], 10} = 4224960 $ & $ \Gamma_{[8, (2^5),0], 13} = -2824394880 $ \\
$ \Gamma_{[6, (2^6),0], 5} = -2000 $ & $ \Gamma_{[7, (2^6),0], 8} = -1226256 $ & $ \Gamma_{[8, (2^6),0], 11} = 906723840 $ \\
$ \Gamma_{[6, (2^7),0], 3} =  448 $ & $ \Gamma_{[7, (2^7),0], 6} = 348054 $  & $ \Gamma_{[8, (2^7),0], 9} =  -287936880 $ \\
$ \Gamma_{[6, (2^8),0], 1} = -96 $ & $ \Gamma_{[7, (2^8),0], 4} = -96256 $  & $ \Gamma_{[8, (2^8),0], 7}  =  90364160 $   \\
& $ \Gamma_{[7, (2^9),0], 2} = 25820 $  & $ \Gamma_{[8, (2^9),0], 5} = -27996424 $  \\
& $ \Gamma_{[7, (2^{10}),0], 0} =-6672 $  & $ \Gamma_{[8, (2^{10}),0], 3} = 8551776 $  \\
& & $ \Gamma_{[8, (2^{11}),0], 1} = -2571612 $ \\
\hline
\end{tabular}
\end{table}

\begin{table}[ht]
\centering
\begin{tabular}{|l|}
\noalign{\smallskip}
\hline
$d=10, \beta=0$ \\
\hline
$ \Gamma_{[10, (3^5),0], 14} = -276649331840 $ \\
$ \Gamma_{[10, (3^6),0], 11} = 12995931360 $ \\
$ \Gamma_{[10, (3^7),0], 8} =  559349440 $ \\
$ \Gamma_{[10, (3^8),0], 5}  =   -21525168  $ \\
$ \Gamma_{[10, (3^9),0], 2} =   -713472 $ \\
\hline
\end{tabular}
\end{table}

\begin{table}[ht]
\centering
\caption{Open Gromov-Witten invariants of a $\CP^2$ blowup  at $r$ real points and $s$ pairs of conjugate points.}
\begin{tabular}{|l|l|}
\hline
$d=6$ & $d=7$ \\
\hline
$ \Gamma_{[6, 0, (2^4)], 1} = -12 $ &  $\Gamma_{[7, (3), (2^4)], 1} = 48 $ \\
$ \Gamma_{[6, (2^2), (2^3)], 1} = -20 $ &  $\Gamma_{[7,0,(2^5)],0} = 48 $ \\
$ \Gamma_{[6, (2^4], (2^2)], 1} = -36 $ &  $\Gamma_{[7, (2^2), (2^4)], 0} = -48 $ \\
$ \Gamma_{[6, (2^6), (2)], 1} = -60 $ &  $\Gamma_{[7, (2^4), (2^3)], 0} = -384 $ \\
$ \Gamma_{[6, (2^8), 0], 1} =  -92 $ &   $ \Gamma_{[7, (2^6), (2^2)], 0} = -1216 $\\
$ \Gamma_{[6, 0, (2^3)], 5} = -156  $ & $ \Gamma_{[7, (2^8), (2)], 0} = -3056 $ \\
$ \Gamma_{[6, (2^2), (2^2)], 5} = -472 $ & $ \Gamma_{[7, (2^{10}), 0], 0} = -6672 $ \\
$ \Gamma_{[6, (2^4), (2)], 5} = -1044 $ & \\
$ \Gamma_{[6, (2^6), 0], 5} = -2000 $ &  \\
\hline
\end{tabular}
\end{table}

\pagebreak

Recall that the open Gromov-Witten invariant $\Gamma_{[d,\al,\beta],k} $ is equivalent to the Welschinger invariant counting curves of degree
\( d [\Line]- \sum_i [\ExDiv_i] - \sum_j ([\ExDiv_{r+j}] + [\ExDiv_{r+j+s}]) \)
passing through $k$ real points and $l=\half( 3d -\abs{\al} -2\abs{\beta} -k-1) $ complex conjugate pairs of points by the following relation
\begin{align*}
\Gamma_{[d,\al,\beta],k} = \pm 2^{1-l} W_{d [\Line]-\sum_i [\ExDiv_i] - \sum_j ([\ExDiv_{r+j}] + [\ExDiv_{r+j+s}]), l}.
\end{align*}
A precise formula for the sign is given in Section~\ref{initial calculations}.

\vspace{.5 cm}
\noindent
Institute of Mathematics \\
Hebrew University, Givat Ram \\
Jerusalem, 91904, Israel \\

\end{document}